\documentclass[12pt]{article}
\usepackage{amsmath,amsthm}
\usepackage{amssymb,multirow,booktabs}
\usepackage{amsfonts,mathrsfs}
\usepackage{bbm}
\usepackage[colorlinks,anchorcolor=blue,citecolor=blue]{hyperref}

\allowdisplaybreaks

\theoremstyle{plain}
\newtheorem{thm}{\noindent\bf Theorem}[section]
\newtheorem{cor}{\noindent\bf Corollary}[section]

\newtheorem{lem}{\noindent\bf Lemma}[section]
\newtheorem{prop}{\noindent\bf Proposition}[section]
\theoremstyle{remark}
\newtheorem{rmk}{\noindent \bf \textit{Remark}}[section]

\newcommand{\D}{\displaystyle}
\newcommand{\non}{\nonumber}
\renewcommand{\P}{\mathbb{P}}
\newcommand{\E}{\mathbb{E}}
\newcommand{\ddr}{\mathrm{d}}

\usepackage{xcolor}
\newcounter{count}
\usepackage{enumerate}
\usepackage{multicol}
\numberwithin{equation}{section}

\newcommand{\kara}{Proposition \ref{prop:kara}}
\newcommand{\ktt}{Proposition \ref{prop:ktt}}
\newcommand{\asp}{\hyperlink{asp}{\color{black}$\mathbb{H}$}}

\begin{document}

\title{Behaviors near explosion of nonlinear CSBPs with regularly varying mechanisms}

\author{Bo Li\thanks{School of Mathematics and LPMC, Nankai University, Tianjin 300071, PR China. E-mail: libo@nankai.edu.
cn}\,, Cl\'ement Foucart\thanks{CMAP, Ecole Polytechnique, IP Paris, Palaiseau, France. Email: clement.foucart@polytechnique.edu } \ and Xiaowen Zhou\thanks{Department of Mathematics and Statistics, Concordia University, Montreal, Canada H3G 1M8. Email:
xiaowen.zhou@concordia.ca}}

\maketitle

We study the explosion phenomenon of nonlinear continuous-state branching processes (nonlinear CSBPs). First an explicit integral test for explosion is designed when the rate function does not increase too fast. We then exhibit three different regimes of explosion when the branching mechanism and the rate function are regularly varying respectively at $0$ and $\infty$  with indices  $\alpha$ and $\beta$ such that $0\leq\alpha\leq \beta$ and $\beta>0$. If $\alpha>0$ then the renormalisation of the process before its explosion is linear. When moreover $\alpha\neq \beta$, the limiting distribution is that of a ratio of two independent random variables whose laws are identified. When $\alpha=\beta$, the limiting random variable shrinks to a constant. Last, when $\alpha=0$, i.e. the branching mechanism is slowly varying at $0$, the process is studied with the help of a nonlinear renormalisation. The limiting distribution becomes the inverse uniform distribution. This complements results known in the case of finite mean and provides new insight for the classical explosive continuous-state branching processes (for which $\beta=1$).


\section{Introduction and main results}\label{sec:intro}

Continuous-state branching processes (CSBP for short) are processes representing the total size of a population in which, intuitively, each individual reproduces independently with the same law. It is well known that when the offspring distribution has infinite first moment, the process might explode in finite time, in the sense that large jumps happen so fast that the process reaches infinity in  finite time; see for instance  Savits \cite{Savits} and Grey \cite{Grey}. 

The purpose of this article is to study the renormalisations in law of the processes prior to their explosion. Instead of only considering classical CSBPs, we treat the larger class of the so-called nonlinear CSBPs. The latter have been introduced and studied in recent works; see e.g. Li and Zhou \cite{LiZ2021} and the references therein. They are parametrized by two functions:
a branching mechanism $\psi$, which is the Laplace exponent of a spectrally positive L\'evy process (SPLP for short) $\xi:=(\xi(t),t\geq 0)$, and
a rate function $R$, which is assumed to be a \textit{positive} measurable function on $(0,\infty)$, bounded away from both $0$ and $\infty$ on every compact subset of $(0,\infty)$. Equivalently, $1/R$ is measurable positive and locally bounded.
Heuristically, the branching mechanism $\psi$ governs the reproduction, as in the classical CSBP, and $R$ controls the reproduction rate.

More explicitly, let $x\in [0,\infty)$ and denote by $\mathbb{P}_x$ the law of the SPLP $\xi$ with $\xi(0)=x$ a.s.
We will also write $\mathbb{P}:=\mathbb{P}_0$. For any $a,b\geq 0$, denote the first passage times of the process $\xi$ by
\begin{equation}\label{firstpassagetimexi}
\tau_{a}^{-}:=\inf\{t>0, \xi(t)<a\}
\quad\text{and}\quad
\tau_{b}^{+}:=\inf\{t>0, \xi(t)>b\}.
\end{equation}
Define the additive functional $\eta$ by
\begin{equation}\label{defn:eta}
\eta(t):=\int_{0}^{t}\frac{\ddr s}{R(\xi(s))}
\quad\text{and}\quad
\eta^{-1}(s):=\inf\{t\ge0: \eta(t)>s\},
\end{equation}
for $t\in[0,\tau_{0}^{-})$ and $s\in[0,\eta(\tau_{0}^{-}))$ with convention $\inf\emptyset=\infty$. The nonlinear CSBP $X$, with branching mechanism $\psi$ and rate function $R$ is defined by time changing $\xi$ as follows: for any $t\geq 0$,
\begin{equation}
\label{defn:X}
X(t):=\left\{
\begin{aligned}
&\,\,\xi\big(\eta^{-1}(t)\big), &&\text{when}\,\, t\in [0, \eta(\tau^-_0)),\\
&\,\,\infty, &&\text{when}\,\, \eta(\infty)\leq t<\tau^-_0=\infty,\\
&\,\,0, &&\text{when}\,\, \tau^-_0<\infty, \, \eta(\tau^-_0)<\infty\,\,\text{and}\,\,
	t\geq \eta(\tau^-_0).
\end{aligned}\right.
\end{equation}
We stress that by definition the process $X$ takes its values in  $[0, \infty]$ and both boundaries $0$ and $\infty$ are absorbing.

In order to avoid cumbersome notation, we also denote the law of the process $X$ started from $x$ by $\mathbb{P}_x$. This will not cause any confusion. The assumption that $1/R$ is positive on compact subsets of $(0,\infty)$ ensures  that the time-changed SPLP is well-defined when started from a point $x\in (0,\infty)$. We also highlight that for any rate function $R$, the nonlinear CSBP $X$ inherits from $\xi$ the strong Markov property and its c\`adl\`ag sample paths.

The definition \eqref{defn:X} above is also referred to as a Lamperti-type transform of the process $\xi$. We call $\xi$ the {\it parent L\'evy process} of $X$. It is well known that, if the function $R$ is the identity function, the process $X$ reduces to the classical CSBP with branching mechanism $\psi$, and if $R$ is an exponential function, the process $X$ is associated to a positive self-similar Markov process. The two aforementioned transforms are called the \textit{Lamperti transformations} in the literature; see e.g. Kyprianou \cite[Chapter 12 and Chapter 13]{Kyprianou2014book}. We also mention that the polynomial case, for which $R(x)=x^{\beta}$, has been introduced in Li \cite{Li2019}.

A nice feature of this class of Markov processes is that the function $R$ allows for a new  degree of freedom to exhibit richer long-term behaviors (extinction, coming down from infinity, explosion, etc) in comparison to the classical CSBP framework.
For instance, Li et al. \cite{LTZ22} investigate the role of the rate function in the extinguishing phenomenon i.e. the convergence of $X$ towards $0$ without reaching it. Foucart et al. \cite{FLZ2021} treat the case of a parent L\'evy process that either oscillates or drifts towards $-\infty$ and address the question of whether the time-changed L\'evy process $X$ can be started from infinity, and  at what speed it comes down from infinity. 
Last but not least, the study of the explosion of the process $X$ and of its behavior ``one second" prior to explosion has been carried out in \cite{LiZ2021} in the case of a parent L\'evy process with finite mean. The latter assumption is in particular ruling out the classical CSBPs. 

Our main objective is to investigate the case a parent L\'evy process with \textit{infinite} mean. We first collect some elementary properties of nonlinear CSBPs that are consequences of the time-change relationship \eqref{defn:X}. We review then briefly some results known in the setting of a jump measure with finite mean. We shall also explain  why other arguments have to be designed when it is infinite. 

From now on, consider a parent L\'evy process $\xi$ with no negative jumps drifting towards $+\infty$. For any $y>0$ and $a>0$, set
$$T^{+}_{b}:=\inf\{t>0: X_{t}> b\}  \text{ and } T^{-}_{a}:=\inf\{t>0: X_{t}\leq a\}.$$
By the assumption of local boundedness of $R$ on $(0,\infty)$, we have for any $b>x>a>0$, $\mathbb{P}_x$-almost surely 
\[
\{T_b^{+}\wedge T_{a}^{-}<\infty\}=\{\tau_{b}^{+}\wedge \tau_{a}^{-}<\infty,\eta(\tau_{b}^{+}\wedge \tau_{a}^{-})<\infty\}=\{\tau_{b}^{+}\wedge \tau_{a}^{-}<\infty\},
\]
and the latter event has probability one.
Define the explosion time of $X$ by \[T_\infty^{+}:=\inf\{t>0: X_t=\infty\}=\underset{b\rightarrow \infty}{\lim} \!\! \uparrow T^{+}_{b}.\] 
The following identity relating $T^{+}_{\infty}$ for $X$ and $\eta$ for $\xi$ is obtained immediately from the definition \eqref{defn:X}: almost surely on the non-extinction event, that is, on
$\{X_t \underset{t\rightarrow \infty}{\nrightarrow} 0\}=\{\tau_0^{-}=\infty\}$,
we have
\begin{equation}\label{perpetualintegral}
T_{\infty}^{+}=\eta(\infty)=\int^{\infty}_0\frac{\ddr s}{R(\xi_s)}.
\end{equation}
Hence, the event of explosion $\{T_\infty^{+}<\infty\}$ is equivalent to the event of finiteness of the \textit{perpetual integral} $\eta(\infty)$. This is a classical topic, of independent interest, which has been studied intensively. We refer the reader for instance to  D\"oring and Kyprianou \cite{Doring2015}, Kolb and Savov \cite{Kolb2020}, Li and Zhou \cite{LPSZ2022} and  the references therein.

We now sketch the results obtained in \cite{LiZ2021} under the assumption that the parent L\'evy process $\xi$ has a finite mean, i.e.
\begin{equation}\label{finite-mean}
	\mathbb{E}[\xi(1)]=-\psi'(0)\in(0,\infty),
\end{equation}
and explain the differences with the  setting $\mathbb{E}[\xi(1)]=\infty$. It is established in \cite{LiZ2021} that the process explodes on the non-extinction event with positive probability, i.e. $\mathbb{P}_x(\eta(\infty)<\infty; \tau_0^{-}=\infty)>0$, see \eqref{perpetualintegral}, if and only if
\begin{equation}\label{conditionfinitemean}
\int^{\infty}\frac{\ddr y}{R(y)}<\infty.
\end{equation}
Explosion occurs here not due to the jumps  but due to drastic speed-up of the time, i.e. $R(x)$ goes quickly towards $\infty$ as $x\to\infty$, so that the long time behavior of $\xi$ becomes the short time behavior of $X$. This can also be explained by the strong law of large numbers for L\'evy processes, which ensures that for any $x\geq 0$,  $\mathbb{P}_x$-almost surely
\[\xi(s)\underset{s\rightarrow \infty}{\sim} s\mathbb{E}[\xi(1)]=-s\psi'(0).\]
Replacing $\xi(s)$ in the integrand of \eqref{perpetualintegral} by the equivalent function $-s\psi'(0)$ would then lead to the condition \eqref{conditionfinitemean}.

In the same fashion, and still at a heuristic level, on the event of non-extinction, the nonlinear CSBP $X$ fluctuates, as $t$ converges to $T^{+}_{\infty}$, along the explosive deterministic path $(x(t))_{t\geq 0}$ satisfying
\[
\ddr x(t)=-\psi'(0)\cdot R(x(t))\,\ddr t,\quad x(0)>0,
\]
and the speeds of explosion, i.e. the behavior of $X(T^{+}_{\infty}-t)$, for different types of rate functions, have been investigated in \cite{LiZ2021}. The techniques designed in \cite{LiZ2021} are mostly relying on the fact that the process $\xi$ has a stationary overshoot distribution under condition (\ref{finite-mean}). Since the latter does not exist when $-\psi'(0)=\infty$, most of the arguments explained above break down, including the question of explosion itself (i.e. the finiteness of $\eta(\infty)$). Moreover, there is no  deterministic path clearly defined around which the process  oscillates before its explosion. The problems of finding a sharp condition for explosion and how a nonlinear CSBP behaves near its explosion time in the case of infinite mean (hence with very large jumps) require therefore a study on their own and are the principal aims of this paper.
\medskip

Our first result is a necessary and sufficient condition (in terms of an integral test) for explosion of nonlinear CSBPs when the rate function $R$ is subject to a certain assumption on its growth. Note that no assumption is made on the mean in the following theorem. For any monotone function $H$, we denote by $\ddr H$ the associated Stieljes measure\footnote{We will also use the notation $H(\ddr z)$}: $\int_x^y\ddr H(z)=H(y)-H(z)$ for $x<y$.
\begin{thm}[Explosion criterion]\label{thmexplosiong1}

If $R$ is increasing and
	\begin{equation}\label{eqn:R1}
		\limsup_{x\to\infty}\frac{R(x+a)}{R(x)}<\infty\quad\text{for every $a>0$},
	\end{equation}
	then for any $x\in (0,\infty)$,
	\begin{equation}\label{explosionequiv}
	\P_{x}(T^{+}_{\infty}<\infty)>0
	\text{ if and only if }
	\int^{\infty}\frac{-1}{\psi(1/y)}\ddr \Big(-\frac{1}{R(y)}\Big)<\infty.
	\end{equation}
	Moreover, when \eqref{explosionequiv} holds, one has $\mathbb{P}_x(T_\infty^+<\infty|T_a^{-}=\infty)=1$ for any $a\in (0,x)$.
\end{thm}
\begin{rmk}
By condition \eqref{eqn:R1}, function $R$ cannot be rapidly varying. For instance, the exponential functions are excluded.
\end{rmk}
\begin{rmk}\label{rmk:1.1}
The setting of a classical CSBP is covered by the theorem. Indeed the identity function $R(y)=y$, for all $y\geq 0$, satisfies \eqref{eqn:R1} and the integral condition in \eqref{explosionequiv} can be rewritten, with the help of change of variable, $u=1/y$, as $\int_0\frac{\ddr u}{-\psi(u)}<\infty$. We recover here the classical condition for explosion of a CSBP; see Grey \cite{Grey} and the forthcoming Section \ref{sec:CSBP} for more details.
\end{rmk}

The study of the behavior prior to explosion  requires to characterize the asymptotic overshoot  $X(T^{+}_{x})-x$ as $x\to \infty$. In the case of infinite mean, the latter tends to $\infty$ in probability and we have to look for some explicit scaling limits. We then need to impose some assumption of regular variations on both the rate function and the distribution of the large jumps of the parent process $\xi$.

Let $\mathcal{R}_{\beta}$ be the set of regularly varying functions with parameter $\beta$; see Section \ref{RVfunction} for more details on those functions. In the next statements, we shall mainly focus on nonlinear CSBPs with a branching mechanism and a rate function satisfying the following regular variation properties:
\begin{center}
\hypertarget{asp}{$\mathbb{H}$:}
$\quad -\psi\in\mathcal{R}_{\alpha}$ at $0$
and $R\in\mathcal{R}_{\beta}$ at $\infty$ for some $\alpha\in[0,1], \beta>0$.
\end{center}
For $\alpha\in[0,1)$, we always have $\psi'(0)=-\infty$. The case $\alpha=1$ is a \textit{critical} case for which both $-\psi'(0)=\infty$ and $-\psi'(0)\in (0,\infty)$ are possible. Notice also that any function $R$ regularly varying near $\infty$ satisfies \eqref{eqn:R1}. We will see in the next corollary that the integral test in Theorem \ref{thmexplosiong1}  can  be further simplified in this setting. We then study the explosion behavior according to differential values of parameters $\alpha$ and $\beta$ when assumption \asp \ holds.
\begin{cor}\label{cor:1}
If $R\in\mathcal{R}_{\beta}$ at $\infty$ for some $\beta>0$,
then for any $x\in (0,\infty)$, we have
\begin{equation}
\label{explosionconditionpsiR}
\P_{x}(T^{+}_{\infty}<\infty)>0
	\text{ if and only if }
\mathcal{E}:=\int^{\infty}\frac{-\ddr y}{yR(y)\psi(1/y)}<\infty.
\end{equation}
	
If assumption \asp\ is in force, we further have
	\begin{enumerate}
		\item if $\alpha>\beta>0$, then $\P_{x}\big(T_{\infty}^{+}<\infty\big)=0$;
		\item if $\alpha=\beta>0$, then $\P_{x}\big(T_{\infty}^{+}<\infty\big)>0$ if and only if $\mathcal{E}<\infty$;
		\item if $\beta>\alpha\ge0$, then $\P_{x}\big(T_{\infty}^{+}<\infty\big)>0$.
	\end{enumerate}
\end{cor}
\begin{rmk} The three cases in Corollary \ref{cor:1} cover different situations:
\begin{enumerate}
\item When $1\geq\alpha>\beta>0$, the rate function $R$ is slowing down the jump rate so much that the explosion is prevented. The process $X$ is in this case transient and nonexplosive. We shall rule out this case in the next.
\item In the critical case $\alpha=\beta$, the power functions in $R(y)$ and $\psi(1/y)$ cancel out and whether explosion occurs or not depends on the slowly varying parts in $R$ and $\psi$. For instance, if $\psi(s)=s\log s$ (Neveu's branching mechanism) and $R(y)=yL(y)$ for some slowly varying function $L$ at $\infty$, then the process explodes if and only if $\mathcal{E}=\int^{\infty}\frac{\ddr y}{yL(y)\log(y)}<\infty$.
\item When $\beta>\alpha\geq 0$, the explosion occurs. More precisely, when $\beta> 1$, the rate function itself can cause the explosion since  $\int^\infty \frac{\ddr y}{R(y)}<\infty$. When $\beta \le1$ (the classical CSBP with $\beta=1$ falls into this case), the rate function itself ``can not bring" the process to $\infty$ in finite time since $\int^\infty \frac{\ddr y}{R(y)}=\infty$, and the parent process is necessary to have extremely large jumps for explosion to occur. 
\end{enumerate}
\end{rmk}
\medskip
We now present our main results on
the ``speeds of explosion" for the nonlinear CSBPs satisfying the  assumption \asp.  We will study the asymptotics of $X(T^{+}_\infty-t)$ when $t$ goes to $0+$, given that $T^{+}_\infty<\infty$. Different renormalisations that lead to convergence in law or convergence in probability
will be obtained. The limit law of $X(T^{+}_\infty-t)$ after renormalisation will be expressed with the help of the following two random variables:
\begin{itemize}
	\item For $\alpha\in (0,1)$, let $\chi_{\alpha}$ be a random variable taking values in $[1,\infty)$ with probability law
	\begin{equation}\label{defn:chi}
		\P(\chi_{\alpha}>z)=\frac{\sin\alpha\pi}{\pi}\int_{0}^{1}u^{\alpha-1}(z-u)^{-\alpha} \ddr u\quad\text{for $z\ge1$}.
	\end{equation}
	We also set $\chi_\alpha=1$ a.s. for $\alpha=1$ and $\chi_\alpha=\infty$ a.s. for $\alpha=0$.
	\item For $\alpha\in [0,1]$ and $\beta>\alpha$, let $\varrho_{\alpha,\beta}$ be a nonnegative random variable with moment generating function given by
	\begin{equation}\label{defn:rho}
		 M_{\varrho_{\alpha,\beta}}(\theta):=\E[e^{\theta\varrho_{\alpha,\beta}}]=1+\sum_{n\ge1}\frac{\theta^{n}}{n!}\Big(\prod_{k=1}^{n}\frac{\Gamma(k(\beta-\alpha)+1)}{\Gamma(k(\beta-\alpha)+\alpha)}\Big)\quad\text{ for } |\theta|<\frac{1}{\beta}.
	\end{equation}
When $\alpha=1$, $\varrho_{\alpha,\beta}\equiv 1$ a.s. and when $\alpha=0$, $\varrho_{\alpha,\beta}$ is exponentially distributed with parameter $1/\beta$. It is also notable that when $\beta=1$ and $0\leq \alpha<1$, \eqref{defn:rho} can be simplified into
\[M_{\varrho_{\alpha,1}}(\theta)=\sum_{n\geq 0}\frac{\theta^{n}}{n!}\Gamma(1+n(1-\alpha)),\]
and we recover the moment-generating function of the Weibull distribution with parameter $\frac{1}{1-\alpha}$, namely the random variable $\rho_{\alpha,1}$ has cumulative distribution function
\[\mathbb{P}(\rho_{\alpha,1}\leq t)=1-e^{-t^{\frac{1}{1-\alpha}}}, \text{  for all }t\geq 0.\]
\end{itemize}
The laws of $\chi_\alpha$ and $\varrho_{\alpha,\beta}$ have already appeared in the literature and can be expressed with the help of a stable subordinator through, respectively, its renormalized asymptotic overshoot and a certain perpetual integral; see the forthcoming Propositions \ref{prop:overshoot} and \ref{prop:rho} for more details.\\

From now on, we shall always work on the event of  explosion. We write $\P_{x}(\cdot|T_{\infty}^{+}<\infty)$ for the conditional distribution given $X(0)=x$ and $T_{\infty}^{+}<\infty$. The convergence in distribution is denoted by ``$\Rightarrow$''. Note that when explosion happens, we always have $\mathcal{E}<\infty$ and thus
\begin{equation}\label{defn:s}
	\varphi(x):=\int_{x}^{\infty}\frac{-\ddr y}{yR(y)\psi(1/y)}<\infty\quad\text{for $x>1/p$},
\end{equation}
where $p:=\inf\{s>0: \psi(s)>0\}$ is the largest zero of the branching mechanism (if $\psi(s)<0$ for all $s>0$, $p=\infty$ and $1/p=0$). Background on the function $\psi$ is provided in Subsection \ref{ssec:2.1}. Note also that the inverse function of $\varphi$, \[\varphi^{-1}(t):=\inf\{x>0, \varphi(x)=t\},\] is well defined for any $t<\varphi(p^{-1})$.
\\

\begin{thm}[Speeds of explosion]\label{thm:speed}
Assume that \asp\ holds and $\mathcal{E}<\infty$ in \eqref{explosionconditionpsiR}.
We have the following regimes of weak convergence for different choices of $\beta\geq \alpha\geq 0$.
\begin{enumerate}
\item \label{case1} If $\beta>\alpha>0$, then \begin{equation}
\label{linearspeed1}\frac{X(T_{\infty}^{+}-t)}{\varphi^{-1}(t)}\bigg|_{\P_{1}(\cdot|T_{\infty}^{+}<\infty)}
\Rightarrow \frac{1}{\chi_{\alpha}}\times\varrho_{\alpha,\beta}^{\frac{1}{\beta-\alpha}}\quad\text{as $t\to 0+$},
 \end{equation}
where $\chi_{\alpha}$ and $\varrho_{\alpha,\beta}$ are two independent random variables whose probability laws are specified in \eqref{defn:chi} and \eqref{defn:rho}, respectively.
\item \label{case2}
If $\beta=\alpha>0$ and $\mathcal{E}<\infty$, then
\[\frac{X(T_{\infty}^{+}-t)}{\varphi^{-1}(\Gamma(\alpha) t)}\Big|_{\P_{1}(\cdot|T_{\infty}^{+}<\infty)}
\rightarrow 1\text{ in probability as $t\to0+$}.
\]
\item\label{case3}
If $\beta>\alpha=0$, then
\begin{equation}\label{nonlinearspeed}
\frac{\psi(1/X(T^{+}_{\infty}-t))}{\psi(1/\varphi^{-1}(t))}\bigg|_{\P_{1}(\cdot|T_{\infty}^{+}<\infty)}
\Rightarrow \frac{1}{\Lambda}\quad\text{as $t\to 0+$},
\end{equation}
where $\Lambda$ is a uniform random variable on $(0,1)$.
\end{enumerate}
\end{thm}
\begin{rmk}
The order of the renormalisation, which we call the speed of explosion, is different in each of the three regimes in Theorem \ref{thm:speed}.
\begin{enumerate}
	\item For $\beta>\alpha>0$, one has $\varphi\in\mathcal{R}_{\alpha-\beta}$ at $\infty$, so that $\varphi^{-1}\in\mathcal{R}_{\frac{1}{\alpha-\beta}}$ at $0$. The speed of explosion is thus of the order $t^{\frac{1}{\alpha-\beta}}$.

	\item For $\beta=\alpha>0$, one can easily check that the function $\varphi$ is slowly varying at $\infty$. Hence, its inverse $\varphi^{-1}$ is rapidly varying; see the next remark for some examples.
		\item For $\beta>\alpha=0$, the convergence in \eqref{linearspeed1} actually also holds  true. But since $\chi_\alpha=\infty$ a.s., the limit is $0$.
The speed of explosion of $X$ when $\alpha=0$ is thus faster than $\varphi^{-1}$.
The nonlinear renormalisation in \eqref{nonlinearspeed} depends on the form of the slowly varying function $-\psi$. For instance, if $R(x)\sim x^{\beta}$ and $\psi(1/x)\sim -(1/\log x)^r$ as $x\to\infty$ for some $r>0$, then one can check from \eqref{nonlinearspeed} that
\[
\frac{\log X(T^{+}_{\infty}-t)}{\log 1/t}\Big|_{\P_{1}(\cdot|T_{\infty}^{+}<\infty)} \Rightarrow \frac{1}{\beta \Lambda^{1/r}}\quad\text{as $t\to0+$.}
\]
Equivalently,
\begin{equation}\label{loglog}
\log \log X(T^{+}_{\infty}-t)- \log \log 1/t \, \Big|_{\P_{1}(\cdot|T_{\infty}^{+}<\infty)} \Rightarrow \log 1/\beta+\mathbbm{e}_r
\end{equation}
where $\mathbbm{e}_r$ is an exponential random variable with parameter $r$ and we have used the fact that
$\mathbbm{e}_r\overset{D}{=}\log \frac{1}{\Lambda^{1/r}}$.
\end{enumerate}
\end{rmk}
\begin{rmk}
Classical CSBPs, for which $R(y)=y$ and hence $\beta=1$, can belong to any of the three different cases covered by Theorem \ref{thm:speed} according the behaviour of their branching mechanism $\psi$ near $0$. By the change of variable $u=1/y$ in \eqref{defn:s}, one has $\varphi(x)=\int_0^{1/x}\frac{\ddr u}{-\psi(u)}$ for all $x>1/p$.
\begin{enumerate}
\item In Case \ref{case1}, since $\rho_{\alpha,1}$ has a Weibull distribution with parameter $\frac{1}{1-\alpha}$, the random variable $\rho_{\alpha,1}^{\frac{1}{1-\alpha}}$ has a standard exponential law and the limit law in \eqref{linearspeed1} is a mixture of exponential laws with random parameter $\chi_\alpha$. This case is  explored further in Section \ref{sec:CSBP}.
\item Case \ref{case2} with $\alpha=\beta=1$ contains for instance the branching mechanism $\psi$ such that $\psi(x)\underset{x\rightarrow 0}{\sim} -x(\log 1/x)^\gamma$ with $\gamma>1$. Simple calculations give
$\varphi^{-1}(t) \underset{t\rightarrow 0}{\sim} e^{c_\gamma t^{-\frac{1}{\gamma-1}}}$
with $c_\gamma=(\gamma-1)^{-\frac{1}{\gamma-1}}>0$.
\item In Case \ref{case3}, the function $\varphi^{-1}$ can behave differently according to the slowly varying function $\psi$. However, using the fact that $-\psi$ is increasing in a neighbourhood of $0$, one can plainly check that
$\varphi(x)\geq -\frac{\psi(1/x)}{x}=:g(x)$. Since
$\varphi^{-1}(t)\geq g^{-1}(t)$,
we have
\[\frac{X(T_{\infty}^{+}-t)}{g^{-1}(t)}\Big|_{\P_{x}(\cdot|T_{\infty}^{+}<\infty)}
\rightarrow 0 \text{ in probability as $t\to0+$}.\]
Notice that $g^{-1}\in \mathcal{R}_{-1}$ at $0$. So, it is of order $1/t$. Lastly, note that the first term in the limit \eqref{loglog} vanishes in the CSBP case, since $\beta=1$, and we have an exponential law with parameter $r$.
\end{enumerate}
\end{rmk}

\begin{rmk} In Case \ref{case1} with $\alpha=1$, our result also holds in some cases with finite mean and coincide with the result obtained in \cite[Theorem 3.6-(a)]{LiZ2021}. Indeed when $-\psi\in\mathcal{R}_{1}$ and $\E[\xi_{1}]=-\psi'(0)<\infty$, we have $x\psi(1/x)\to-\psi'(0)<\infty$ as $x$ goes to $\infty$. Thus, the explosive condition in \eqref{explosionconditionpsiR} reduces to \eqref{conditionfinitemean}. Moreover, 
$\varphi(x)\underset{x\rightarrow \infty}{\sim} \frac{1}{\E[\xi_{1}]}\int_{x}^{\infty}\frac{\ddr y}{R(y)}$
and in both Case \ref{case1} and Case \ref{case3} of Theorem \ref{thm:speed} for $\alpha=1$ and $-\psi'(0)<\infty$, we have
\[
\frac{X(T_{\infty}^{+}-t)}{\varphi^{-1}(t)}\rightarrow 1\text{ in probability as $t\to0+$}.
\]
No classical branching process falls into this setting.
\end{rmk}

The rest of the paper is organized as follows. In Section \ref{sec:preliminaries}, we gather fundamental facts on the spectrally positive L\'evy processes, their scale functions and their potential measure. We then recall elements on regularly varying functions and state some results about the overshoot of a SPLP with a regularly varying Laplace exponent. A representation of the random variable $\rho_{\alpha,\beta}$ with the help of a stable subordinator is also given. We prove these results as well as Theorems \ref{thmexplosiong1} and \ref{thm:speed} in Section \ref{proofs}.	\\

\noindent \textbf{Notation}: Throughout the paper, we  write $\D f(x)\sim g(x)$ if $f(x)/g(x)\to1$, $f(x)=o(g(x))$ if $f(x)/g(x)\to0$, $f(x)=O(g(x))$ if $\limsup_{x\to\infty}f(x)/g(x)<\infty$ and finally $f(x)\asymp g(x)$ if $f(x)/g(x)$ is bounded away from both $0$ and $\infty$. For any Borel measure $\nu$ on $(0,\infty)$
we write $\bar{\nu}(z):=\int_{z}^{\infty}\nu(\ddr y)$ for all $z>0$. We denote by $\overset{\text D}{=}$ the equality in distribution. Throughout the paper we take the convention $\inf \emptyset=\infty$.

\section{Preliminaries}\label{sec:preliminaries}
\subsection{Spectrally positive L\'evy processes and their scale functions}\label{ssec:2.1}
Let $\xi$ be a spectrally positive L\'evy process, i.e. a L\'evy process with no negative jumps. Its Laplace exponent is well defined and takes the L\'evy-Khintchine form, for $s\ge0$,
\begin{equation}\label{defn:psi}
\psi(s)=\log\E\big[e^{-s\xi(1)}\big]=\gamma s+\frac{\sigma^{2}}{2}s^{2}+\int_{0}^{\infty}\big(e^{-s z}-1+s z\mathbf{1}(z\le1)\big)\nu(\ddr z),
\end{equation}
where $\gamma\in\mathbb{R}, \sigma\ge0$ are constants and $\nu$ is the L\'evy measure, namely, it is a $\sigma$-finite measure on $(0,\infty)$ satisfying $\int_{0}^{\infty}(1\wedge z^{2})\nu(\ddr z)<\infty$.
\begin{itemize}
\item When $\xi$ is not a subordinator, one has $\psi(\infty)=\infty$ and the process $\xi$ may in general oscillates or drifts towards +$\infty$ or $-\infty$. As we are interested in the explosion, we shall focus on the case of a process $\xi$ drifting towards $+\infty$. This is equivalent to the assumption $\mathbb{E}(\xi(1))=-\psi'(0)\in (0,\infty]$. In the non-subordinator case, the right-continuous inverse of $\psi$ defined by $\phi(q):=\inf\{s>0, \psi(s)>q\}$
is well-defined for all $q>0$. Recall also that we set $p=\phi(0)$. Another equivalent condition for the process to drift towards $+\infty$ is $p>0$.
\item When $\xi$ is a subordinator, i.e. $\xi$ has monotone nondecreasing sample paths, the function $\psi$ is non-positive and takes the form
\begin{equation}\label{defn:psisub}
\psi(s)=-\delta s+\int_{0}^{\infty}\big(e^{-s z}-1\big)\nu(\ddr z),
\end{equation}
with $\delta\geq 0$ is the drift and the measure $\nu$ satisfies $\int_0^{\infty}(1\wedge z)\nu(\ddr z)<\infty$. In this setting, $\psi(\infty)=-\infty$ and it is more convenient to work with the function $-\psi$ which satisfies $-\log \mathbb{E}[e^{-s\xi(1)}]=-\psi(s)$.
In the subordinator case, $\phi$ is not defined on $\mathbb{R}^{+}$ and in particular $p=\phi(0)=\inf\emptyset=\infty$.
\end{itemize}
The study of the explosion of the process $X$ is more involved when the parent L\'evy process $\xi$ is not a subordinator, i.e. $0<p<\infty$. This is mainly due to the fluctuations of the process which makes it more difficult to follow its position and to compute its potential measure.\\

\textbf{Non-subordinator case: $0<p<\infty$}. In order to handle the non-subordinator case, we shall need some properties of fluctuation theory. Most of the literature on fluctuations of spectrally one-sided L\'evy processes concerns the spectrally negative cases, that is to say, processes with no positive jumps. We rewrite below some classical fluctuation identities in the spectrally positive setting.  They  can all be found using the fact that if $\xi$ is a SPLP then its dual process $-\xi$ is spectrally negative.

Scale function $W$  plays a key role in the fluctuation theory and will be one of our main tool later on. It is defined as the unique continuous increasing function on $\mathbb{R}^{+}$ satisfying
\begin{equation}\label{defLTscalefunction}
\int_{0}^{\infty}e^{-sz}W(z) \ddr z=\frac{1}{\psi(s)},\quad\text{for $s>p$},
\end{equation}
with the convention $W(z)=0$ for $z<0$. We refer the readers to \cite{Bertoin96book} and \cite{Kyprianou2014book} for more detailed discussions on the spectrally negative L\'evy process and its scale functions. We shall also use the notation $W(\ddr z)$ for the Stieljes measure associated to $W$.

Recall the first passage times of $\xi$, $\tau_a^-$ and $\tau_b^+$, see \eqref{firstpassagetimexi}. For any $x\in (a,b)$, one has the identities
\begin{equation}\label{exitprobability}
\mathbb{P}_x(\tau_a^{-}\leq \tau_b^{+})=\frac{W(b-x)}{W(b-a)}
\quad\text{and}\quad
\mathbb{P}_x(\tau_a^{-}<\infty)
=\underset{b\rightarrow \infty}{\lim} \mathbb{P}_x(\tau_a^{-}\leq \tau_b^{+})
=e^{-p(x-a)}.
\end{equation}
The following expression of the potential measure of $\xi$ killed on leaving the interval $[a,\infty)$ can be found for instance in \cite[Corollary 8.8]{Kyprianou2014book}. For any $a\in\mathbb{R}$ and $x,y\in(a,\infty)$,
\begin{equation}\label{eqn:resl}
U_{a}(x,\ddr y):= \int_{0}^{\infty}\P_{x}( \xi(t)\in \ddr y, t<\tau_{a}^{-}) \ddr t
= \big(e^{p(a-x)}W(y-a)-W(y-x)\big) \ddr y,
\end{equation}
for any $x\geq a$ and where we recall that $p=\phi(0)$. Define for any $y\in \mathbb{R}$
\begin{equation}\label{defn:kappa}
\kappa(y):=e^{py}\phi'(0)-W(y).
\end{equation}
Recall that $W(y)=0$ for $y<0$ so that $\kappa(y)=e^{py}\phi'(0)$ for $y<0$. One has
\begin{equation}\label{eqn:resl1}
U_{a}(x,\ddr y)
=\big(\kappa(y-x)-e^{p(a-x)}\kappa(y-a)\big) \ddr y.
\end{equation}
Similarly, under the assumption $p>0$, the process is transient and its potential measure has a density bounded on $[0,\infty)$, see \cite[Corollary 8.9]{Kyprianou2014book}. When $\xi$ starts from $0$, the latter is given by
\begin{equation}\label{defn:Unonsub}
U(\ddr y):=\int_{0}^{\infty}\P( \xi(t)\in \ddr y) \ddr t
= \kappa(y)\ddr y\end{equation}
and we call $\kappa$ the \textit{potential density} of $\xi$.

The function $\kappa$ is not explicit in general. There is however an equivalence near $0$ of its Laplace transform.
\begin{lem}\label{lem:k}
Suppose that $0<p<\infty$. We have as $s\to0+$
\begin{equation}\label{eqn:hk}
\hat{\kappa}(s):=\int_{0}^{\infty}e^{-sy}\kappa(y)\ddr y
=\int_{0}^{\infty}e^{-sy}U(\ddr y)
\sim \frac{-1}{\psi(s)}.
\end{equation}
\end{lem}
\begin{proof}
One can calculate directly that
\[
\hat{\kappa}(s)=\int_{0}^{\infty}e^{-sy}\kappa(y) \ddr y=\frac{\phi'(0)}{s-p}-\frac{1}{\psi(s)}
=\left(1-\frac{\phi'(0)\psi(s)}{s-p}\right)\frac{1}{-\psi(s)}
\sim \frac{-1}{\psi(s)}\text{\ at $0$},
\]
where the identity above holds first for $s>p$, and then for all $s>0$ by analytic extension and the equivalence follows readily since $\psi(0)=0$.
\end{proof}

\textbf{Subordinator case:
$p=\infty$}. Assume now that $\xi$ is a subordinator. For any $x\geq 0$, define the potential measure \begin{equation}\label{eqn:potentialsub} U(x,\ddr y):=\int_0^{\infty}\mathbb{P}_x(\xi(t)\in \ddr y)\ddr t \text{  and } U(\ddr y):=U(0,\ddr y).
\end{equation}
One has, see e.g. \cite[Chapter III, page 74]{Bertoin96book}, for all $s>0$,
\begin{equation}\label{defn:Usub}
\hat{U}(s):=\int_{0}^{\infty}e^{-sy}U(\ddr y)=\mathbb{E}\left[\int_0^{\infty}e^{-s\xi(t)}\ddr t\right]=\frac{1}{-\psi(s)}.
\end{equation}
Moreover, it is easily checked that
\[\hat{U}(x,s):=\int_x^{\infty}e^{-sy}U(x,\ddr y)=e^{-sx}\int_0^\infty e^{-sz}U(\ddr z) \underset{s\rightarrow 0+}{\sim}\hat{U}(s).
\]
We stress that contrary to the previous setting, the measure $U$ may not have a density. However, in both cases the measure on $(0,\infty)$, $\mathbf{1}(y>0)U(\ddr y)$ has a Laplace transform satisfying \eqref{eqn:hk}.
\subsection{Regularly varying functions and Laplace exponents}\label{RVfunction}

A function $f$ is \textit{regularly varying} at $\infty$ with index $\gamma\in\mathbb{R}$, written   $f \in \mathcal{R}_{\gamma}$ at $\infty$,
if it is a positive function defined on a neighbourhood $[A,\infty)$ of $\infty$ satisfying
\[
f(\lambda x)\big/f(x)\to \lambda^{\gamma}\quad\text{as $x\to\infty$ for all $\lambda>0$}.
\]
In particular, $f$ is called \textit{slowly varying} at $\infty$ if $\gamma=0$.
Similarly, a function $f$ is regularly varying at $0$ with index $\gamma$ if and only if $x\mapsto f(1/x)\in\mathcal{R}_{-\gamma}$ at $\infty$.
 Note that $f\in\mathcal{R}_{\gamma}$ if and only if $x\mapsto x^{-\gamma}f(x)\in\mathcal{R}_{0}$.
	
Many properties of regularly varying functions will be used in the proofs of the main results. Specifically, we shall appeal to the uniform convergence theorem (UCT for short), see Bingham et al. \cite[Theorem 1.5.1]{Bingham1987regular},
the Karamata and Tauberian theorems, see \cite[Theorem 1.5.11 and Theorem 1.7.1]{Bingham1987regular} and the monotone density theorem, \cite[Theorem 1.7.2]{Bingham1987regular}.
We recall the Karamata theorem and the Tauberian theorem for convenience of the reader.

\begin{prop}[Karamata's Theorem]\label{prop:kara}
Let $f\in\mathcal{R}_{\rho}$ at $\infty$ for $\rho\in \mathbb{R}$ be locally bounded on $[A,\infty)$.
\begin{enumerate}
\item For any $\sigma\ge-(1+\rho)$, we have
\[
x^{\sigma+1}f(x)\bigg/\int_{A}^{x}t^{\sigma}f(t) \ddr t\underset{x\rightarrow \infty}{\to} 1+\rho+\sigma;
\]
\item For any $\sigma<-(1+\rho)$ (and for $\sigma=-(\rho+1)$ if $\int^{\infty}t^{-\sigma}f(t) \ddr t<\infty$), we have
\[
x^{\sigma+1}f(x)\bigg/\int_{x}^{\infty}t^{\sigma}f(t) \ddr t\underset{x\rightarrow \infty}{\to} -\big(1+\rho+\sigma\big).
\]
\end{enumerate}
\end{prop}

\begin{prop}[Karamata's Tauberian Theorem]\label{prop:ktt}
Let $\mu$ be a positive measure on $\mathbb{R}^{+}$ and $l\in\mathcal{R}_{0}$. Put  $\mu(x):=\mu(0,x]$ for $x>0$. Then for $c,\rho\ge0$,
\[
\mu(x)\sim \frac{c x^{\rho}l(x)}{\Gamma(1+\rho)}
\quad\text{if and only if}\quad
\hat{\mu}(s):=\int_{0}^{\infty}e^{-sy}\mu(dy)\sim c s^{-\rho}l(1/s)
\]
as  $x\to\infty$  and  $s\to0+$ respectively.
For $c=0$ the left hand side is interpreted as $\mu(x)=o(x^{\rho}l(x))$ and the right hand side is interpreted similarly.
\end{prop}

The following lemmas are consequences of the propositions above.
The first one guarantees that if $f\in\mathcal{R}_{-\gamma}$ with negative index, then it is equivalent to a decreasing and differentiable function.

\begin{lem}\label{lem:f-}
If $f\in\mathcal{R}_{-\gamma}$ at $\infty$ for some $\gamma>0$, we have
\begin{equation}
\label{eqn:f}
f(x)\sim f^{\downarrow}(x):=\gamma\int_{x}^{\infty}\frac{f(y)}{y}\,\ddr y
\quad\text{as $x\to\infty$}.
\end{equation}
Notice that the function $f^{\downarrow}$ is decreasing.
\end{lem}

The next lemma allows us, among other things, to construct a regularly varying Laplace exponent $\psi$  from the L\'evy measure $\nu$ in \eqref{defn:psi}. Recall that $\bar{\nu}$ denotes the tail of $\nu$.

\begin{lem}
Suppose that $-\psi'(0+)=\infty$. For any $\alpha \in [0,1]$,
\begin{equation}\label{eqn:nu}
-\psi(s)\sim s^{\alpha}l(1/s)\text{\ at $0+$\quad if and only if\quad}
\int_{1}^{x}\bar{\nu}(z) \ddr z\sim \frac{x^{1-\alpha}l(x)}{\Gamma(2-\alpha)}\in\mathcal{R}_{1-\alpha}
\text{\ at $\infty$},
\end{equation}
where $l\in\mathcal{R}_{0}$ at $\infty$.

For $\alpha\in[0,1)$, \eqref{eqn:nu} is equivalent to $\D \bar{\nu}(x)\sim \frac{x^{-\alpha}l(x)}{\Gamma(1-\alpha)}\in\mathcal{R}_{-\alpha}$ at $\infty$.\end{lem}
\begin{proof}
It follows  from \eqref{defn:psi} that when $\E[\xi(1)]=-\psi'(0)=\infty$,
\[
-\psi(s)\sim \int_{1}^{\infty}(1-e^{-sy})\nu(\ddr y)= s\int_{0}^{\infty}e^{-sz}\bar{\nu}(1\vee z) \ddr z, \text{ as } s\to0+.
\]
 Applying \ktt\ we have \eqref{eqn:nu}.
If $\alpha\in[0,1)$, we further have from \eqref{eqn:nu} that
\begin{center}
$\D \int_{1}^{x}\bar{\nu}(z) \ddr z\sim \frac{x^{1-\alpha}l(x)}{\Gamma(2-\alpha)}$
at $\infty$ \quad if and only if\quad
$\D \bar{\nu}(x)\sim \frac{x^{-\alpha}l(x)}{\Gamma(1-\alpha)}\in\mathcal{R}_{-\alpha}$ at $\infty$,
\end{center}
applying the monotone density theorem.
 \end{proof}
Note that if $-\psi \in \mathcal{R}_\alpha$ then $\psi$ is negative in a neighbourhood of $0$ and necessarily $p\in (0,\infty]$, so that the L\'evy process $\xi$, with exponent $\psi$,  drifts towards $+\infty$ a.s.. Another important application of the Tauberian theorem is the regularly varying property for the integral $\int_0^{x}U(\ddr y)$ with $U$ the potential measure of $\xi$; see \eqref{defn:Unonsub} and \eqref{defn:Usub}.

\begin{lem}[Regular variation of the potential measure]
Suppose $-\psi \in \mathcal{R}_\alpha$ at $0$. Then in both cases $0<p<\infty$ and $p=\infty$, we have
\begin{equation}
\label{eqn:k:1}
\int_0^{x}U(\ddr y)\sim \frac{-1}{\psi(1/x)}\frac{1}{\Gamma(1+\alpha)}\in\mathcal{R}_{\alpha}\text{\ at $\infty$}.
\end{equation}
\end{lem}
\begin{proof}
The result in the subordinator case is a direct consequence of \ktt. In the non-subordinator case, one has $\int_0^{x}U(\ddr y)=\int_{0}^{x}\kappa(y) \ddr y$ and by \eqref{eqn:hk}, $
\hat{\kappa}(s)\sim \frac{-1}{\psi(s)}\in\mathcal{R}_{-\alpha}\text{\ at $0$}$. Thus, \eqref{eqn:k:1} follows from \ktt.
\end{proof}

\subsection{On the probability laws of $\chi_\alpha$ and $\varrho_{\alpha,\beta}$}
To study the explosive behaviour of $X$, it is necessary to understand the asymptotic distribution of the overshoots for the parent process $\xi$. In the case of infinite mean, we need $-\psi$ to be regularly varying to take use of the scaling property. The following result is known for a subordinator, we refer to Bertoin \cite[Exercice III.4]{Bertoin96book};
see also Doney and Maller \cite{Doney2002} for a more detailed discussion on the convergence of $\xi(\tau_{x}^{+})/x\rightarrow 1$  either in probability or almost surely.

\begin{prop}\label{prop:overshoot}
Let $\alpha \in [0,1]$, suppose that $p>0$ and $-\psi\in\mathcal{R}_{\alpha}$ at $0$. Then
\begin{equation}\label{convergencechialpha}
x^{-1}\xi(\tau_{x}^{+})\Rightarrow \chi_{\alpha}
\quad\text{under $\P$ as $x\to\infty$},
\end{equation}
where $\chi_{\alpha}$ is a random variable whose distribution is specified in \eqref{defn:chi}. The convergence \eqref{convergencechialpha} also holds under $\P_{1}(\cdot|\tau_{a}^{-}=\infty)$ for every $a\in(0,1)$.
\end{prop}
The proof of Proposition \ref{prop:overshoot} is provided in Section \ref{sec:proofprop1overshoot}. In the case $\alpha=0$, one observes by combining \cite[Theorem III.6]{Bertoin96book} with Proposition \ref{prop:overshoot} that $\xi(\tau_{x}^{+}-)=o(x)$ and $x=o(\xi(\tau_{x}^{+}))$ in probability when  $x\to\infty$. In this extreme case, when the process upcrosses an arbitrary large level $x$, it overshoots so far above it that $x$ is negligible in comparison to the position of the process after the overshoot.

We shall establish in this direction the following proposition.
\begin{prop}\label{cor:a=0}
If $-\psi\in\mathcal{R}_{0}$ at $0$, we have
\begin{equation}\label{convergenceunif}
\psi(1/\xi(\tau_{x}^{+}))\big/\psi(1/x)\Rightarrow \Lambda,
\quad\text{under $\P$ as $x\to\infty$},
\end{equation}
where $\Lambda$ is a uniform random variable on $(0,1)$.
The convergence \eqref{convergenceunif} also holds under $\mathbb{P}_{1}(\cdot|\tau_{a}^{-}=\infty)$ for every $a\in(0,1)$.
\end{prop}
Such a result is not really new. It appears under another form in  Kl\"uppelberg et al. \cite[Theorem 4.4]{KKM04}. Assumptions of the latter theorem are however not satisfied here and we will give a short proof of the proposition which suits our framework. The nonlinear renormalisation is also reminiscent of that obtained in some functional limit theorems for subordinators with slowly varying exponent; see for instance \cite{MS2019}, \cite{KL2013} and \cite{FY2023}. The proof of Proposition \ref{cor:a=0} is in Section \ref{sec=}. This will follow from arguments designed to treat the Case \ref{case2} of Theorem \ref{thm:speed}.

As mentioned in Section \ref{sec:intro},  the random variable $\varrho_{\alpha,\beta}$, defined in \eqref{defn:rho}, can be encountered in different contexts. The next proposition states  that the law of $\varrho_{\alpha,\beta}$ is indeed that of some perpetual integral for an $\alpha$-stable subordinator.
\begin{prop}\label{prop:rho}
For $\beta>\alpha>0$, let $S:=(S(t),t\geq 0)$ be a stable subordinator started from $S(0)=1$ with Laplace exponent $\psi(s)=-c_{0} s^{\alpha}$ for some $\alpha\in(0,1)$. Then for any $n\geq 1$,
\begin{equation}\label{eqn:rho:mn}
\E_1\bigg[\Big(\int_{0}^{\infty}\frac{\ddr t}{S(t)^{\beta}}\Big)^{n}\bigg]=\frac{1}{c_{0}(\beta-\alpha)}\prod_{k=1}^{n}\frac{\Gamma\big(k(\beta-\alpha)+1\big)}{\Gamma\big(k(\beta-\alpha)+\alpha\big)}.
\end{equation}
Moreover, the following large deviation estimate holds:
\begin{equation}\label{eqn:rho:ldp}
\limsup_{t\to\infty}t^{-\frac{1}{1-\alpha}}
\log\P_{1}\bigg(\int_{0}^{\infty}\frac{\ddr u}{S(u)^{\beta}}>t\bigg)
=-(1-\alpha)c_{0}^{\frac{1}{1-\alpha}}(\beta-\alpha)^{\frac{\alpha}{1-\alpha}}.
\end{equation}
For the case $c_{0}(\beta-\alpha)=1$ we have by \eqref{defn:rho}
\[
\varrho_{\alpha,\beta}\overset{\text D}{=}\int_{0}^{\infty}\frac{\ddr t}{S(t)^{\beta}}.
\]
\end{prop}
\begin{rmk}
The perpetual integral $\int_{0}^{\infty}\frac{\ddr t}{S(t)^{\beta}}$ under $\mathbb{P}_1$ can also be seen as the explosion time of the nonlinear CSBP started from $1$, with rate function $R(x)=x^{\beta}$, and parent L\'evy process the subordinator $S$.
\end{rmk}

\section{Proofs of the main results}\label{proofs}
This section is dedicated to the proofs of our main results.
We first establish the criterion of explosion stated in Theorem \ref{thmexplosiong1}. Corollary \ref{cor:1} on the regularly varying cases  then follows easily.
We next prove Proposition \ref{prop:overshoot} on the asymptotic overshoot distribution of $\xi$ and Proposition \ref{prop:rho} which provides a representation of $\varrho_{\alpha,\beta}$ in terms of a perpetual integral of a stable subordinator.
Then we investigate the residual explosion time $T_{\infty}^{+}-T_{x}^{+}$ for large $x$ by estimating its moments and establish the convergences of $X(T_{\infty}^{+}-t)$  in Theorem \ref{thm:speed} case by case.
\subsection{Proof of Proposition \ref{prop:overshoot} and Proposition \ref{prop:rho}}

\subsubsection{Proof of Proposition \ref{prop:overshoot} (study of the overshoot)}\label{sec:proofprop1overshoot}
Assume first that $\xi$ is not a subordinator. Denote the running supremum of the process $\xi$ by $\bar{\xi}(t):=\sup\{\xi(s), s\le t\}$, and call $H:=(H(t))_{t\geq 0}$ the ladder height process, that is,
\begin{equation}\label{defn:ladder}
H(t)=\bar{\xi}(L^{-1}(t))
\end{equation}
for the local time $L$ of $\bar{\xi}-\xi$ at $0$. By definition of the latter, $L^{-1}(t)$ is a zero of $\bar{\xi}-\xi$; see \cite[Proposition  IV.7(iii)]{Bertoin96book} and one has therefore almost surely for all $t\geq 0$, $$H(t)=\bar{\xi}(L^{-1}(t))=\bar{\xi}(L^{-1}(t))-\xi(L^{-1}(t))+\xi(L^{-1}(t))=\xi(L^{-1}(t)).$$
Moreover, when the process $\xi$ goes above level $x$ for the first time, a new supremum is reached and we thus have
\begin{equation}\label{identitiesovershoot}
\P\big(\bar{\xi}(\tau_{x}^{+}-)\in \ddr u,\xi(\tau_{x}^{+})\in \ddr v\big)=\P\big(H(\sigma_{x}^{+}-)\in \ddr u,H(\sigma_{x}^{+})\in \ddr v\big)
\end{equation}
where $\sigma_{x}^{+}$ denotes the first passage time above $x$ of the process $H$. We are thus left to study the overshoot of $(H(t))_{t\geq 0}$;
see for instance Doney and Maller \cite{Doney2002} and \cite[Lemma 2.5]{LiZ2021} where this technique is also applied.

According to \cite[Theorem VII.4]{Bertoin96book}, $H$ is a subordinator with Laplace exponent given by
\begin{equation}\label{eqn:ladderexponent}
\Phi(s):=\frac{-1}{t}\log\E\big[\exp(-s H(t))\big]=\frac{\psi(s)}{s-p}\quad\text{for all}\quad s>0.
\end{equation}
Even though we do not  need it later,  for completeness we also identify  the drift and the L\'evy measure of $H$. It follows from \eqref{defn:psi} that
\begin{equation}\label{eqn:prop2:5}
\begin{aligned}
\Phi(s)=\frac{\psi(s)-\psi(p)}{s-p}
=&\ \Big(\frac{\sigma^{2}}{2}(s+p)+\gamma\Big)
+\int_{0}^{\infty}\Big(\frac{e^{(p-s)z}-1}{s-p} e^{-pz}+z \mathbf{1}(z<1)\Big)\nu(\ddr z)\\
=&\ \Big(\frac{\sigma^{2}}{2}s+\frac{\sigma^{2}}{2}p+\gamma\Big)
+\int_{0}^{\infty}\int_{0}^{z}\Big(\mathbf{1}(z<1)-e^{-pz+(p-s)y}\Big) \ddr y\nu(\ddr z).
\end{aligned}
\end{equation}
Since $p>0$ and $\psi(p)=0$, we have for $s=0$
\[
0=\Big(\frac{\sigma^{2}}{2}p+\gamma\Big)+\int_{0}^{\infty}\int_{0}^{z}\Big(\mathbf{1}(z<1)-e^{-pz+py}\Big) \ddr y\nu(\ddr z).
\]
Therefore, substituting the above identity into \eqref{eqn:prop2:5} gives
\begin{equation}\label{eqn:prop2:2}
\begin{aligned}
\Phi(s)=\frac{\psi(s)}{s-p}
=&\ \frac{\sigma^{2}}{2}s+ \int_{0}^{\infty}\int_{0}^{z}\Big(e^{-pz+py}-e^{-pz+(p-s)y}\Big)\nu(\ddr z)\ddr y\\
=&\ \frac{\sigma^{2}}{2}s+\int_{0}^{\infty}\big(1-e^{-sy}\big)\Big(\int_{y}^{\infty}e^{p(y-z)}\nu(\ddr z)\Big) \ddr y\\
=:&\ \frac{\sigma^{2}}{2}s+\int_{0}^{\infty}\big(1-e^{-sy}\big)\Pi(\ddr y)
= s\Big(\frac{\sigma^{2}}{2}+ \int_{0}^{\infty}e^{-sz}\bar{\Pi}(z) \ddr z\Big),
\end{aligned}\end{equation}
where $\bar{\Pi}(z):=\int_{z}^{\infty}\Pi(\ddr u)$; see \cite[Exercice 6.5]{Kyprianou2014book}. \\

We next study the limit in law of $\xi(\tau_{x}^{+})/x=H(\sigma_{x}^{+})/x$ as $x$ goes to $\infty$. Recall our assumption $-\psi\in\mathcal{R}_{\alpha}$ at $0$. The formula in \eqref{eqn:ladderexponent} entails that the subordinator $H$ has also a Laplace exponent regularly varying at $0$ with index $\alpha$. Note that the study of scaling limit in such a subordinator setting is the purpose of \cite[Exercice III.3-4]{Bertoin96book}. We provide a solution below.

Denote by $\mathcal{U}$ the potential measure of $H$, that is,
\[
\mathcal{U}(x):=\E\Big[\int_{0}^{\infty}\mathbf{1}\big(H(t)\le x\big)\ddr t\Big].
\]
Then
\[\hat{\mathcal{U}}(s):=\int_{0}^{\infty}e^{-sy}\mathcal{U}(\ddr y)=\frac{1}{\Phi(s)} \text{ for all }s>0;\]
see \cite[Page 74]{Bertoin96book}. According to \cite[Proposition III.2, Theorem III. 5]{Bertoin96book}, one has
\begin{equation}
\label{eqn:prop2:4}
\P\big(H(\sigma_{x}^{+}-)\in \ddr u,H(\sigma_{x}^{+})\in \ddr v\big)
=\mathcal{U}(\ddr u)\Pi(\ddr v-u)
\quad\text{and}\quad
\P\big(H(\sigma_{x}^{+})=x\big)=\frac{\sigma^{2}}{2}u(x)
\end{equation}
where $u$ is the continuous and positive version of the density of $\mathcal{U}$ when $\sigma>0$. Similarly to \cite[Lemma III.7]{Bertoin96book}, the following calculations follow from \eqref{eqn:prop2:4}.
\begin{align}\label{eqn:prop2:3}
\int_{0}^{\infty}e^{-qt}\E\big[e^{-\theta x^{-1} \xi(\tau_{xt}^{+})}\big] \ddr t
=&\ \iiint_{u+v>xt\ge u}e^{-qt-\frac{\theta}{x}(u+v)}\mathcal{U}(\ddr u)\Pi(\ddr v) \ddr t+\int_{0}^{\infty}e^{-(q+\theta)t}\frac{\sigma^{2}}{2}u(xt) \ddr t\non\\
=&\ \frac{1}{q}\iint_{u,v>0}\mathcal{U}(\ddr u)\Pi(\ddr v)e^{-\frac{\theta}{x}(u+v)}\big(e^{-\frac{q}{x}u}-e^{-\frac{q}{x}(u+v)}\big)+\frac{\sigma^{2}}{2x}\hat{\mathcal{U}} \left(\frac{q+\theta}{x}\right)\non\\
=&\ \frac{1}{q}\hat{\mathcal{U}}\Big(\frac{q+\theta}{x}\Big)\left(\Phi\Big(\frac{q+\theta}{x}\Big)-\Phi\Big(\frac{\theta}{x}\Big)-\frac{\sigma^{2}}{2}\frac{q}{x}\right)+\frac{\sigma^{2}}{2x}\hat{\mathcal{U}}\big(\frac{q+\theta}{x}\big)\non\\
=&\ \frac{1}{q}\Big(1-\frac{\Phi(\theta/x)}{\Phi((\theta+q)/x)}\Big).
\end{align}
For $\alpha\in(0,1)$,
\[
\lim_{x\to\infty}\int_{0}^{\infty}e^{-qt}\E\Big[e^{-\theta x^{-1} \xi(\tau_{xt}^{+})}\Big] \ddr t
= \frac{1}{q}\Big(1-\Big(\frac{\theta}{\theta+q}\Big)^{\alpha}\Big)
= \int_{0}^{\infty}e^{-qt}\Big(\int_{t}^{\infty}\frac{\theta^{\alpha}u^{\alpha-1}}{\Gamma(\alpha)}e^{-\theta u} \ddr u\Big) \ddr t.
\]
Further note that $x^{-1}\xi(\tau_{xt}^{+})$ is increasing in $t>0$.
So, the function $\E\big[e^{-\theta x^{-1}\xi(\tau_{xt}^{+})}\big]$ is decreasing in $t$.
We then obtain by integration by parts that
\[\lim_{x\to\infty}\int_{0}^{\infty}e^{-qt}\ddr \bigg(1-\E\Big[e^{-\theta x^{-1} \xi(\tau_{xt}^{+})}\Big]\bigg)
=\int_{0}^{\infty}e^{-qt}\ddr \Big(\int_{0}^{t}\frac{\theta^{\alpha}u^{\alpha-1}}{\Gamma(\alpha)}e^{-\theta u} \ddr u\Big).
\]
which implies that for a.e. $t>0$, and then for all $t>0$ by continuity,
\[
\lim_{x\to\infty}\left( 1-\E\Big[e^{-\theta x^{-1} \xi(\tau_{xt}^{+})}\Big]\right)
=\int_{0}^{t}\frac{\theta^{\alpha}u^{\alpha-1}}{\Gamma(\alpha)}e^{-\theta u} \ddr u.
\]
In particular, for $t=1$
\[\begin{aligned}
\lim_{x\to\infty}&\E\Big[1-e^{-\theta x^{-1} \xi(\tau_{x}^{+})}\Big]
=\theta \times \theta^{\alpha-1}\int_{0}^{1}\frac{u^{\alpha-1}}{\Gamma(\alpha)}e^{-\theta u} \ddr u\\
=&\ \theta \int_{0}^{\infty}\frac{v^{-\alpha}e^{-\theta v}}{\Gamma(1-\alpha)}\ddr v
\int_{0}^{1}\frac{u^{\alpha-1}e^{-\theta u}}{\Gamma(\alpha)} \ddr u
= \theta \int_{0}^{\infty}e^{-\theta z} \ddr z\int_{0}^{1\wedge z}\frac{u^{\alpha-1}(z-u)^{-\alpha}}{\Gamma(\alpha)\Gamma(1-\alpha)} \ddr u\\
=&\ 1-e^{-\theta}+ \int_{1}^{\infty}\theta e^{-\theta z}\ddr z\int_{0}^{1}\frac{u^{\alpha-1}(z-u)^{-\alpha}}{\Gamma(\alpha)\Gamma(1-\alpha)} \ddr u
=\int_0^{\infty}\theta e^{-\theta z}\mathbb{P}(\chi_\alpha>z)\ddr z
\end{aligned}\]
where we recall \eqref{defn:chi}, the identity $\frac{1}{\Gamma(\alpha)\Gamma(1-\alpha)}=\frac{\sin \alpha \pi}{\pi}$
and we make use of the facts that
\[ \int_{0}^{1\wedge z}\frac{u^{\alpha-1}(z-u)^{-\alpha}}{\Gamma(\alpha)\Gamma(1-\alpha)} \ddr u=1 \text{ and } \mathbb{P}(\chi_\alpha>z)=1, \,\, \forall z<1.
\]

The desired result is proved. The results for the degenerate cases of $\alpha=0$ or $1$ can also be derived from \eqref{eqn:prop2:3}.

If $\xi$ is a subordinator, i.e. $-\psi$ takes the form \eqref{defn:psisub}, one can directly obtain the result from the same calculation in \eqref{eqn:prop2:3} by replacing in \eqref{eqn:prop2:4}, $\Phi$, $\mathcal{U}$ and the constant $\sigma^{2}/2$,  respectively, with the Laplace exponent of $\xi$, $-\psi$, its potential measure $U$ in \eqref{eqn:potentialsub} and its drift $\delta$.
Heuristically, in the subordinator case, the ladder height process is simply  the process itself.

\medskip

We now verify that the convergence in law $x^{-1}\xi(\tau_{x}^{+})\Longrightarrow \chi_\alpha$ as $x$ goes to $\infty$ also holds under $\P_{1}(\cdot|\tau_{a}^{-}=\infty)$ for every $1>a>0$. Let $f$ be bounded continuous function, write
\begin{equation}\label{uselessbutnicer}
\E_{1}\big[f\big(\xi(\tau_{x}^{+})/x\big); \tau^{-}_{a}=\infty\big]
=\E_{1}\big[f\big(\xi(\tau_{x}^{+})/x\big)\big]
-\E_{1}\big[f\big(\xi(\tau_{x}^{+})/x\big); \tau^{-}_{a}<\infty\big].
\end{equation}
By the Markov property at time $\tau^{-}_{a}$,
\begin{align}\label{checkovershoot-nonextinction}
&\E_{1}\big[f(\big(\xi(\tau_{x}^{+})/x\big); \tau^{-}_{a}<\infty\big]\nonumber\\
&=\ \E_{1}\big[f\big(\xi(\tau_{x}^{+})/x\big); \tau^{-}_{a}<\tau_{x}^{+}, \tau^{-}_{a}<\infty\big]
+ \E_{1}\big[f\big(\xi(\tau_{x}^{+})/x\big); \tau_{x}^{+}<\tau^{-}_{a}<\infty\big]\nonumber\\
&=\ \E_{a}\big[f\big(\xi(\tau_{x}^{+})/x)\big)\big]\cdot\P_{1}\big(\tau_{a}^{-}<\tau_{x}^{+}\big)+\E_{1}\big[f\big(\xi(\tau_{x}^{+})/x\big); \tau_{x}^{+}<\tau_{a}^{-}<\infty\big]\\
&\underset{x\rightarrow \infty}{\longrightarrow} \E\big[f(\chi_\alpha)\big]\cdot\P_{1}(\tau_{a}^{-}<\infty)\nonumber,
\end{align}
where we use the convergence  $x^{-1}\xi(\tau^{+}_{x})\Longrightarrow \chi_\alpha$ previously established, the fact that $\tau_{x}^{+}\underset{x\rightarrow \infty}{\rightarrow}\infty$ a.s. and the boundedness of $f$. The latter in particular entails that the second term in \eqref{checkovershoot-nonextinction} vanishes as it is bounded as follows.
\[
\big \lvert \E_{1}\big[f\big(\xi(\tau_{x}^{+})/x\big); \tau_{x}^{+}<\tau_{a}^{-}<\infty\big] \big \lvert \leq ||f||_{\infty}\cdot\P_{1}(\tau_{x}^{+}<\tau_{a}^{-}<\infty)
\le ||f||_{\infty}\cdot\P_{x}(\tau_{a}^{-}<\infty)=||f||_{\infty}\cdot e^{p(a-x)}.\]
Going back to \eqref{uselessbutnicer}, we get indeed
\[
\E_{1}\big[f\big(\xi(\tau_{x}^{+})/x\big)\big| \tau^{-}_{a}=\infty\big]
\underset{x\rightarrow \infty}{\longrightarrow}
\E\big[f(\chi_\alpha)\big].\]

\subsubsection{Proof of Proposition \ref{prop:rho} (study of $\varrho_{\alpha,\beta}$)}
\begin{proof}
Let $S$ be an $\alpha$-stable subordinator started from $1$ with $\psi(s)=-c_{0}s^{\alpha}$ for some $\alpha\in(0,1)$. Applying the second Lamperti transform, see \cite[Theorem 13.1]{Kyprianou2014book}, we have
\[
S(t)=\exp\big( \zeta\big(I^{-1}(t)\big)\big)
\]
where by a direct calculation one can show that $\zeta$ is an SPLP started from $0$, with Laplace exponent
\begin{equation}\label{eqn:rho1}
\Psi(s):= -\log\E\big[e^{-s \zeta(1)}\big]
=\frac{c_{0}s}{\Gamma(1-\alpha)}\int_{0}^{\infty}e^{-sz}\big(e^{z}-1\big)^{-\alpha}\,dz=c_{0}\frac{\Gamma(s+\alpha)}{\Gamma(s)}
\end{equation}
and where $I(t):=\int_{0}^{t}e^{\alpha \zeta(s)}\,\ddr s$ and $I^{-1}$ denotes its inverse function. Then, for $\beta>\alpha$, we have
\[
\iota(\infty):=\int_{0}^{\infty}S(t)^{-\beta}\,\ddr t
=\int_{0}^{\infty}e^{-\beta \zeta(I^{-1}(t))}\,\ddr t
=\int_{0}^{\infty}e^{-\beta \zeta(s)}\,\ddr I(s)
=\int_{0}^{\infty}e^{(\alpha-\beta) \zeta(s)}\,\ddr s.
\]
It is proved by Bertoin and Yor  \cite{Bertoin2005} that, for all $r>0$
\[
b_{r}:=\E_1\big[\iota(\infty)^{r}\big]
=\frac{r}{\Psi(r(\beta-\alpha))}\E_1\big[\iota(\infty)^{r-1}\big]
=\frac{r}{\Psi(r(\beta-\alpha))} b_{r-1}.
\]
In particular, taking $r=n$ and making use of \eqref{eqn:rho1} gives the expression \eqref{eqn:rho:mn} for the moment generating function of $\iota(\infty)$.

We are now going to study the tail behaviour of $\iota(\infty)$. Let $\gamma_{0}(1-\alpha)=1$ and $z_{0}=(1-\alpha)c_{0}^{\gamma}(\beta-\alpha)^{\alpha\gamma}$. We will show that
\begin{equation}\label{eqn:rho:mgf}
\E_1\Big[\exp\Big(z\iota(\infty)^{\gamma}\Big)\Big]=
\left\{\begin{array}{ll}
<\infty & \text{if $\gamma<\gamma_{0}$};\\
<\infty & \text{if $\gamma=\gamma_{0}$ and $z<z_{0}$};\\
=\infty & \text{if $\gamma=\gamma_{0}$ and either $z>z_{0}$ or $\gamma>\gamma_{0}$}.
\end{array}\right.
\end{equation}

To prove \eqref{eqn:rho:mgf}, we make use of the asymptotic equivalence
\begin{equation}
\label{eqn:rho4}
\frac{b_{r}}{b_{r-1}}=
\frac{r}{\Psi\big(r(\beta-\alpha)\big)}=\frac{1}{c_{0}(\beta-\alpha)}\frac{\Gamma\big(r(\beta-\alpha)+1\big)}{\Gamma\big(r(\beta-\alpha)+\alpha\big)}\sim
\frac{r^{1-\alpha}}{c_{0}(\beta-\alpha)^{\alpha}}
\quad\text{as $r\to\infty$};
\end{equation}
see e.g. \cite[6.1.47]{Abramowitz1988} and the following version of H\"older's inequality
\begin{equation}\label{eqn:rho2}
(b_{p})^{\frac{q-r}{p-r}}(b_{r})^{\frac{p-q}{p-r}}
\ge b_{\frac{p(q-r)+r(p-q)}{p-r}}=b_{q}
\quad\text{for all $p\ge q\ge r>1$}.
\end{equation}
For each fixed $\gamma>0$, denoting  $m=1+[\gamma]\ge2$ we have from \eqref{eqn:rho4} that
\[
a_{n}:=\E_1\big[(\iota(\infty))^{n\gamma}\big]
=b_{n\gamma}
=\prod_{k=0}^{m-2}\frac{b_{n\gamma-k}}{b_{n\gamma-k-1}}b_{n\gamma-m+1}
=\prod_{k=0}^{m-1}\frac{b_{n\gamma-k}}{b_{n\gamma-k-1}}b_{n\gamma-m}.
\]
Taking $(p,q,r)=(n\gamma-m+1, (n-1)\gamma, n\gamma-m)$ in \eqref{eqn:rho2},
we have from he above that
\[\begin{aligned}
a_{n}=&\ \Big(\prod_{k=0}^{m-2}\frac{b_{n\gamma-k}}{b_{n\gamma-k-1}}b_{n\gamma-m+1}\Big)^{\frac{q-r}{p-r}}
\Big(\prod_{k=0}^{m-1}\frac{b_{n\gamma-k}}{b_{n\gamma-k-1}}b_{n\gamma-m}\Big)^{\frac{p-q}{p-r}}\\
=&\ \prod_{k=0}^{m-2}\frac{b_{n\gamma-k}}{b_{n\gamma-k-1}}\Big(\frac{b_{n\gamma-m+1}}{b_{n\gamma-m}}\Big)^{1+\gamma-m}
\big(b_{p}\big)^{\frac{q-r}{p-r}}\big(b_{r}\big)^{\frac{p-q}{p-r}}\\
\ge&\ \prod_{k=0}^{m-2}\frac{b_{n\gamma-k}}{b_{n\gamma-k-1}}\Big(\frac{b_{n\gamma-m+1}}{b_{n\gamma-m}}\Big)^{1+\gamma-m} b_{q}
= \prod_{k=0}^{m-2}\frac{b_{n\gamma-k}}{b_{n\gamma-k-1}}\Big(\frac{b_{n\gamma-m+1}}{b_{n\gamma-m}}\Big)^{1+\gamma-m} a_{n-1}.
\end{aligned}\]
Making use of \eqref{eqn:rho4} and the fact that  $\gamma,m$ are constants, we have as $n\to\infty$
\begin{equation}\label{eqn:04}
\frac{a_{n}}{a_{n-1}}\ge \prod_{k=0}^{m-2}\frac{b_{n\gamma-k}}{b_{n\gamma-k-1}}\Big(\frac{b_{n\gamma-m+1}}{b_{n\gamma-m}}\Big)^{1+\gamma-m}
\sim n^{(1-\alpha)\gamma}\cdot\Big(\frac{\gamma^{1-\alpha}}{c_{0}(\beta-\alpha)^{\alpha}}\Big)^{\gamma}.
\end{equation}
On the other hand,
taking $(p,q,r)=(n\gamma,n\gamma-1,(n-1)\gamma)$ in \eqref{eqn:rho2},
we have
\[
b_{n\gamma-1}
=b_{q}\le (b_{p})^{\frac{\gamma-1}{\gamma}}(b_{r})^{\frac{1}{\gamma}}
=b_{n\gamma}\cdot\big(a_{n-1}/a_{n}\big)^{1/\gamma}
\]
which together with \eqref{eqn:rho4} gives
\begin{equation}\label{eqn:05}
\frac{a_{n}}{a_{n-1}}
\le\bigg(\frac{b_{n\gamma}}{b_{n\gamma-1}}\bigg)^{\gamma}
\sim n^{(1-\alpha)\gamma}\cdot\Big(\frac{\gamma^{1-\alpha}}{c_{0}(\beta-\alpha)^{\alpha}}\Big)^{\gamma}.
\end{equation}

Combining \eqref{eqn:04} and \eqref{eqn:05} one has that the Taylor expansion
\[
\E\Big[e^{z \iota(\infty)^{\gamma}}\big]
=1+\sum_{n\ge1}\frac{z^{n}}{n!}\E\Big[\iota(\infty)^{n\gamma}\Big]
=\sum_{n\ge0}\frac{a_{n}}{n!}z^{n}
\]
is finite for some $z>0$ only if $\gamma(1-\alpha)\leq 1$. Moreover, for $\gamma=\gamma_{0}=\frac{1}{1-\alpha}$, the series above has a convergence radius of $c_{0}^{\gamma}(\beta-\alpha)^{\alpha\gamma}\gamma_{0}^{-1}=z_{0}$.

Finally, combining Rolski et al. \cite[Theorem 2.3.1]{Rolski1999book} and \eqref{eqn:rho:mgf}, we obtain the large deviation estimate
\[\underset{t\rightarrow \infty}{\limsup}\,t^{-\gamma_0}\log \mathbb{P}(\iota(\infty)>t)=-z_0\]
which gives \eqref{eqn:rho:ldp}.
\end{proof}

\subsection{Proofs of Theorem \ref{thmexplosiong1} and Corollary \ref{cor:1}}
We proceed here with the proof of the integral test for explosion given in \eqref{explosionequiv} (in terms of rate function $R$ and Laplace exponent $\psi$).

\subsubsection{Proof of Theorem \ref{thmexplosiong1} (explosion criterion)}
Recall that the parent L\'evy  process $\xi$ drifts towards $+\infty$, hence the event $\{T_a^{-}=\infty\}$ has positive probability.

Our proof is based on the following equivalences, see \cite[Theorem 2.2]{LPSZ2022} and Kolb and Savov \cite[Theorem 2.1]{Kolb2020}: for any $x>0$,
\begin{align}\label{eqn:explode1-3}
\P_{x}\big(T_{\infty}^{+}<\infty\big)>0&\Longleftrightarrow \P_{x}\big(T_{\infty}^{+}<\infty|T_a^{-}=\infty\big)=1 \text{ for any } a\in (0,x) \nonumber
\\
&\Longleftrightarrow
\mathbb{E}_x\big(\eta(\tau_a^{-})\big)<\infty \text{ for any } a\in (0,x).
\end{align}
Notice that $$\mathbb{E}_x\big(\eta(\tau_a^{-})\big)=U_{a}(1/R)(x)$$ where we recall $U_{a}(x,\ddr y)$ denotes the potential measure of the process killed when it is below level $a$. See \eqref{eqn:resl} for the non-subordinator case. We are going to show that when $R$ is increasing and has a growth satisfying  \eqref{eqn:R1}, the condition $U_{a}(1/R)(x)<\infty$ is equivalent to our more explicit condition \eqref{explosionequiv}.

\medskip

We first treat the simple case where $\xi$ is a subordinator, recalling that in this case $\tau_{a}^{-}=\infty$, $\mathbb{P}_{x}$-a.s. for all $x>a$.
Notice that its Laplace exponent is $-\psi$. One has by definition for any $x>a>0$,
\[
U_{a}\big(1/R\big)(x)=U\big(1/R\big)(x)=\int_{0}^{\infty}\frac{U(\ddr z)}{R(x+z)}
\]
with $U(\ddr z)$ the potential measure of the subordinator $\xi$, see \eqref{defn:Usub}. Recall the condition \eqref{eqn:R1}: $$\limsup_{z\to\infty}\frac{R(z+x)}{R(z)}<\infty \text{ for every $x>0$},$$ and the assumption that $R$ is increasing. One has $\frac{1}{R(z+x)} \asymp \frac{1}{R(z)}$. Hence
$U_{a}\big(1/R\big)(x)<\infty$ if and only if $\int^{\infty}\frac{U(\ddr z)}{R(z)}<\infty$.
Setting $U(y):=U([0,y])$ for any $y\geq 0$, we see, by applying Fubini-Tonelli's theorem, that
\begin{equation*}
\int^{\infty}\frac{U(\ddr z)}{R(z)}<\infty \Longleftrightarrow \int^{\infty}U(y)\ddr\big(-\frac{1}{R(y)}\big)<\infty.
\end{equation*}

By \cite[Proposition III.1]{Bertoin96book}, one has the estimate $U(y)\asymp \frac{1}{-\psi(1/x)}$ and we get the necessary and sufficient condition. The fact that explosion is almost sure is a consequence of \cite[Theorem 2.1]{Kolb2020}.

\medskip

We now assume that $\xi$ is not a subordinator.
The study in this case is more involved as the L\'evy process may go below any level $a$. Define $W_{p}(y):=e^{-py}W(y)$ for any $y\geq 0$. This function is the scale function of the L\'evy process $\xi$ under the new measure $\P^{(p)}$ satisfying $\D\frac{d\P^{(p)}}{d\P}\Big|_{\mathcal{F}_{t}}=e^{-p \xi(t)}$ for all $t\geq 0$; see  e.g. \cite[Chapter 3.3]{Kyprianou2014book}. Note in particular that $W_{p}$ is nondecreasing.
It can be checked directly from \eqref{defLTscalefunction} that
\[
\int_{0}^{\infty}e^{-sy}e^{py}W_{p}(\ddr y)
=\int_{0}^{\infty}e^{-sy}\big(W(\ddr y)-pW(y)\ddr y\big)
=\frac{s-p}{\psi(s)}=\frac{1}{\Phi(s)}=\hat{\mathcal{U}}(s)
\]
for any $s>p$, where $\mathcal{U}$ and $\Phi$ are defined in the proof of  Proposition \ref{prop:overshoot},
and $\hat{\mathcal{U}}$ denotes the Laplace transform of $\mathcal{U}$. Then \cite[Proposition III.1]{Bertoin96book} ensures that for all large $x$
\begin{equation}\label{eqn:explode1-2}
\int_{0}^{x}e^{py}W_{p}(\ddr y)=\mathcal{U}(x)	\asymp \frac{1}{\Phi(1/x)}=\frac{1/x-p}{\psi(1/x)}.
\end{equation}
On the other hand, we have from \eqref{eqn:resl}
\begin{equation}\label{eqn:explode1-1}
\begin{aligned}
e^{px}U_{a}(1/R)(x)
=&\ e^{px}\int_{a}^{\infty}\frac{\ddr y}{R(y)}\big(e^{p(a-x)}W(y-a)-W(y-x)\big)\\
=&\ \int_{a}^{\infty}\frac{e^{py}}{R(y)}\big(W_{p}(y-a)-W_{p}(y-x)\big)\ddr y\\
=&\ \int_{a}^{\infty}\frac{e^{py}\ddr y}{R(y)}\int_{y-x}^{y-a}W_{p}(\ddr z)\\
=&\ \int_{[0,\infty)}e^{pz}W_{p}(\ddr z)\Big(\int_{a}^{x}\frac{e^{py}\ddr y}{R(z+y)}\Big)\\
=&\int_{[0,\infty)}\mathcal{U}(\ddr z)\Big(\int_{a}^{x}\frac{e^{py}\ddr y}{R(z+y)}\Big).
\end{aligned}
\end{equation}
where in the third line we use the fact that $W_p$ is nondecreasing and where in the second to the last equality, we have applied Fubini-Tonelli's theorem and performed a  change of variable. Using the condition \eqref{eqn:R1} and the assumption that $R$ is increasing, it is plain that
\begin{equation}\label{equalence}
\int_{a}^{x}\frac{e^{py}\ddr y}{R(z+y)}\asymp \frac{1}{R(z)}\int_{a}^{x}e^{py}\ddr y
\quad\text{for all large $z$ and fixed $x>a>0$},
\end{equation}
Plugging this into \eqref{eqn:explode1-1}, we see that
\begin{equation}\label{eqn:explode1-4}
U_{a}\big(1/R\big)(x)<\infty
\quad\text{ if and only if}\quad
\int^{\infty}\frac{\mathcal{U}(\ddr z)}{R(z)}<\infty.
\end{equation}
Moreover, by Fubini-Tonelli's theorem and the estimate \eqref{eqn:explode1-2}, we see that
\begin{equation}\label{eqn:explode1-4}
\int^{\infty}\frac{\mathcal{U}(\ddr z)}{R(z)}<\infty
\Longleftrightarrow \int^{\infty}\mathcal{U}(y)\ddr\big(-\frac{1}{R(y)}\big)<\infty \Longleftrightarrow \int^{\infty}\frac{-1}{\psi(1/y)}\ddr\big(-\frac{1}{R(y)}\big)<\infty.
\end{equation}
This finishes the proof.\qed

\subsubsection{Proof of Corollary \ref{cor:1} (application in the regularly varying case)}
For $R\in\mathcal{R}_{\beta}$ at $\infty$, by Lemma \ref{lem:f-} we have
\begin{equation}\label{eqn:explode1-5}
\frac{1}{R(x)}\sim \beta\int_{x}^{\infty}\frac{\ddr z}{zR(z)}
\quad\text{as $x\to\infty$}.
\end{equation}
In addition, \eqref{eqn:R1} holds since $R(x+a)\sim R(x)$ as $x$ goes to $\infty$. Applying the identities in \eqref{eqn:explode1-4} with $1/R$ replaced by the tail integral above proves the first statement of Corollary \ref{cor:1}. The classification of  explosion according to $\alpha$ and $\beta$ when $-\psi\in \mathcal{R}_\alpha$ near $0$ follows plainly by studying whether the test integral is finite or not.\qed

\subsection{Proof of Theorem \ref{thm:speed}}
Until the end of this subsection we focus on the case $\beta\ge\alpha$ and \eqref{explosionconditionpsiR} hold for which explosion can occur.
\subsubsection{Renormalization in law of the explosion time}
\textbf{Case $\beta>\alpha>0$}. The following lemma will be applied to show the weak convergence of $T_{\infty}^{+}$ for $X$ starting from large initial values.
Its proof is a bit technical and is deferred to the Appendix. Recall from Section \ref{sec:preliminaries} that by definition
\[
U_{a}f(x):=\mathbb{E}_{x}\bigg[\int_{0}^{\infty}f(\xi_t)\mathbf{1}(t<\tau_{a}^{-})\ddr t\bigg]
\]
for any bounded measurable function $f$ and any $a>0$.
\begin{lem}[Key equivalence]\label{lem+:uf}
Suppose that $p\in (0,\infty]$, $-\psi\in\mathcal{R}_{\alpha}$ at $0$,
and $f\in\mathcal{R}_{-\gamma}$ at $\infty$ for some $\gamma>\alpha\ge0$. For any $a>0$,
we have
\[U_{a}f(x)\sim \frac{\Gamma(\gamma-\alpha+1)}{\Gamma(\gamma)}
\int_{x}^{\infty}\frac{f(y)}{-y \psi(1/y)}\ddr y
\quad\text{as $x\to\infty$}.\]

\end{lem}
\noindent The next lemma provides a renormalisation in law of the explosion time when the initial value of the process goes to $\infty$.
\begin{lem}\label{lem+:wk}
Suppose that \asp\ hold and $\beta>\alpha\ge0$. We have for any $a>0$,
\[
\frac{\eta(\infty)}{\varphi(x)}
\bigg|_{\P_{x}(\cdot |\tau_{a}^{-}=\infty)}
\Rightarrow \varrho_{\alpha,\beta}
\quad\text{as $x\to\infty$}
\]
where $\varrho_{\alpha,\beta}$ is the random variable whose distribution is characterised by \eqref{defn:rho}.
\end{lem}
\begin{proof}
Recall the following recursive relationship  between the moments of $\eta(\infty)$; see e.g. \cite[Proposition 4.7]{LiZ2021}: for any $n\geq 1$ and $x>a>0$
\begin{equation}\label{recursion}
m_{n}(x):=\ \E_{x}\big[\eta^{n}(\infty); \tau_{a}^{-}=\infty\big]
=n \int_{a}^{\infty}\frac{m_{n-1}(y)}{R(y)}U_{a}(x,\ddr y),
\end{equation}
with $m_{0}(x):=\P_{x}(\tau_{a}^{-}=\infty)=1-e^{p(a-x)}$.
In particular, for every $x>a$ we have
\begin{align}
m_{1}(x):=&\ \E_{x}\big[\eta(\infty); \tau_{a}^{-}=\infty\big]
=\int_{0}^{\infty}f_{1}(y)U_{a}(x,\ddr y),\label{eqn:m1}\\
m_{2}(x):=&\ \E_{x}\big[\eta^{2}(\infty); \tau_{a}^{-}=\infty\big]
=2\int_{0}^{\infty}f_{2}(y)U_{a}(x,\ddr y),\label{eqn:m2}
\end{align}
where for $y>a$, we set
\[
f_{1}(y)=\frac{1-e^{p(a-y)}}{R(y)}\sim\frac{1}{R(y)}\in\mathcal{R}_{-\beta}
\quad\text{as $y\to\infty$ and}\quad
f_{2}(y)=\frac{m_{1}(y)}{R(y)}.
\]

By the assumption $\beta>\alpha$,
Lemma \ref{lem+:uf} can be applied directly to \eqref{eqn:m1} to show that
\[
m_{1}(x)
\sim
 \frac{\Gamma(\beta-\alpha+1)}{\Gamma(\beta)}\int_{x}^{\infty}\frac{-\ddr y}{yR(y)\psi(1/y)}
=: c_{1}\cdot \varphi(x)\in\mathcal{R}_{\alpha-\beta}\text{ at } \infty.
\]
Therefore, $f_{2}\in\mathcal{R}_{\alpha-2\beta}$ at $\infty$ and $2\beta-\alpha>\alpha$.
Applying Lemma \ref{lem+:uf} to \eqref{eqn:m2}, we have
\[\begin{aligned}
m_{2}(x)
\sim&\ \frac{\Gamma(2\beta-2\alpha+1)}{\Gamma(2\beta-\alpha)}
\int_{x}^{\infty}\frac{2\cdot m_{1}(y)\ddr y}{-yR(y)\psi(1/y)}
\sim \frac{\Gamma(2\beta-2\alpha+1)}{\Gamma(2\beta-\alpha)}
\int_{x}^{\infty}\frac{2c_{1}\cdot\varphi(y)\ddr y}{-yR(y)\psi(1/y)}\\
=&\ \frac{\Gamma(2\beta-2\alpha+1)}{\Gamma(2\beta-\alpha)}c_{1}\cdot\varphi^{2}(x)=: c_{2}\cdot\varphi^{2}(x)
\in\mathcal{R}_{2(\alpha-\beta)}\text{ at }\infty.
\end{aligned}\]
Similarly, by induction on $n$ one can also show by  \eqref{recursion} and Lemma \ref{lem+:uf} that for $ n\ge3$,
\[\begin{aligned}
m_{n}(x)
\sim&\ \frac{\Gamma(n(\beta-\alpha)+1)}{\Gamma(n(\beta-\alpha)+\alpha)}
\int_{x}^{\infty}\frac{n\cdot m_{n-1}(y)\ddr y}{-yR(y)\psi(1/y)}
\sim \frac{\Gamma(n(\beta-\alpha)+1)}{\Gamma(n(\beta-\alpha)+\alpha)}
\int_{x}^{\infty}\frac{n c_{n-1}\cdot \varphi^{n-1}(y)\ddr y}{-yR(y)\psi(1/y)}\\
=&\ \frac{\Gamma(n(\beta-\alpha)+1)}{\Gamma(n(\beta-\alpha)+\alpha)}\cdot c_{n-1}\cdot\varphi^{n}(x)=:c_{n}\cdot\varphi^{n}(x)
\in\mathcal{R}_{n(\alpha-\beta)}\text{ at }\infty.
\end{aligned}\]
Since $\P_{x}(\tau_{a}^{-}=\infty)\to1$ as $x$ goes to $\infty$, it follows that
\[
\E_{x}\Big[\Big(\frac{\eta(\infty)}{\varphi(x)}\Big)^{n}\Big| \tau_{a}^{-}=\infty\Big]
\underset{x\rightarrow \infty}{\rightarrow}\prod_{k=1}^{n}\frac{\Gamma(k(\beta-\alpha)+1)}{\Gamma(k(\beta-\alpha)+\alpha)}=:c_{n}
\quad\text{for every $n\in\mathbb{N}$.}
\]

If $\alpha=1$, then $c_{n}\equiv1$ and
\[
\frac{\eta(\infty)}{\varphi(x)}\bigg|_{\P_{x}(\cdot |\tau_{a}^{-}=\infty)}\Rightarrow 1\quad\text{in $L^{2}(\P)$ as $x\to\infty$}.
\]

If $\alpha=0$, then $c_{n}=(\beta-\alpha)^{n}n!=\beta^{n}n!$ and
\[
1+\sum_{n\ge1}\frac{c_{n}}{n!} s^{n}=\sum_{n\ge0}(\beta s)^{n}=\frac{1}{1-\beta s}\quad\text{for $|s \beta|<1$}.
\]

If $\alpha\in(0,1)$, then
\begin{equation}
\label{eqn:rho3}
\frac{c_{n}}{c_{n-1}}=\frac{\Gamma(n(\beta-\alpha)+1)}{\Gamma(n(\beta-\alpha)+\alpha)}\sim (n(\beta-\alpha))^{1-\alpha}\quad\text{as $n\to\infty$},
\end{equation}
see e.g. \cite[6.1.47]{Abramowitz1988}. One can check directly that the sequence $(c_{n})_{n\ge1}$ satisfies the Carleman condition on the moments. The latter determines thus uniquely a positive distribution. Then \cite[Theorem 30.2]{Billingsley2012} can be applied to show the limit.
\end{proof}

\noindent \textbf{Case $1\ge\beta=\alpha>0$:}\label{sec:b=a}
For this critical case, Lemma \ref{lem+:uf} cannot be applied as some integrals may fail to converge; see for example \eqref{eqn+:uf2}.
In the following, an analogue version is derived.
Recalling that for $f\in\mathcal{R}_{-\alpha}$ at $\infty$ and $U_{a}f(x)<\infty$, we have from  Corollary \ref{cor:1}
\[
\int^{\infty}\frac{-f(y)}{y\psi(1/y)}\ddr y<\infty.
\]

\begin{lem}\label{lem=:uf}
Suppose that $p\in (0,\infty]$, $-\psi\in\mathcal{R}_{\alpha}$ at $0$, $f\in\mathcal{R}_{-\alpha}$ at $\infty$ and \eqref{explosionconditionpsiR} hold. Then we have
\[
U_{a}f(x)\sim\frac{1}{\Gamma(\alpha)} \int_{x}^{\infty}\frac{-f(z)}{z\psi(1/z)} \ddr z\in\mathcal{R}_{0}\quad\text{as $x\to\infty$.}\]
\end{lem}
\noindent The proof of Lemma \ref{lem=:uf} is deferred to Appendix. We use it to obtain the following weak convergence of $T_{\infty}^{+}$ for $X$ starting from large values similarly to Lemma \ref{lem+:wk}.

\begin{lem}\label{lem+:wk2}
Suppose that \asp\ holds for $\beta=\alpha>0$ and \eqref{explosionconditionpsiR} holds. Then, for any $a>0$, we have
\[
\frac{\eta(\infty)}{\varphi(x)}
\bigg|_{\P_{x}(\cdot |\tau_{a}^{-}=\infty )}
\Rightarrow \frac{1}{\Gamma(\alpha)}
\quad\text{in $L^{2}(\P)$ as $x\to\infty$.}
\]
\end{lem}
\begin{proof}
By applying Lemma \ref{lem=:uf} to \eqref{eqn:m1}, we find that
\[
m_{1}(x)=U_{a}f_{1}(x)
\sim \frac{1}{\Gamma(\alpha)}\int_{x}^{\infty}\frac{-\ddr y}{zR(z)\psi(1/z)}
=\frac{\varphi(x)}{\Gamma(\alpha)}.
\]
The fact $m_{1}(\infty)=0$ implies $U_{a}f_{2}(x)<\infty$ in \eqref{eqn:m2}.
Further applying Lemma \ref{lem=:uf} gives
\[
m_{2}(x)=2U_{a}f_{2}(x)
\sim \frac{2}{\Gamma(\alpha)}\int_{x}^{\infty}\frac{-m_{1}(z)\ddr z}{zR(z)\psi(1/z)}
\sim\frac{-2}{\Gamma^{2}(\alpha)}\int_{x}^{\infty}\varphi'(z)\varphi(z)\ddr z
=\frac{\varphi^{2}(x)}{\Gamma^{2}(\alpha)}
\]
which proves the assertion.
\end{proof}

We are now ready to proceed with the study of the speeds and  prove  Theorem \ref{thm:speed} for which the next lemma is crucial.
\begin{lem}[Lemma 2.1 in \cite{LiZ2021}]\label{lem:equivsupremum}
Assume that $p>0$ and $\P_{1}(T_{\infty}^{+}<\infty)>0$.
Set $\bar{X}(T_\infty^{+}-t):=\sup\{X(s), s\le T^{+}_\infty-t\}$ for any $t\in [0,T^{+}_\infty]$. One has on the event of explosion
\begin{equation}\label{equivsupremum}
\bar{X}(T_\infty^{+}-t)\underset{t\rightarrow 0}{\sim} X(T_\infty^{+}-t)
\quad\text{under $\P_{1}(\cdot|T_{a}^{-}<\infty)$ a.s..}
\end{equation}
\end{lem}
\begin{proof}
Recall that $\bar{\xi}$ denotes the running supremum process of  $\xi$. It has been established\footnote{the finite mean condition is not used for this result} in \cite[Lemma 2.1]{LiZ2021} that when $p>0$, we have
\[
\xi(t)\big/\bar{\xi}(t)\to1\text{ as $t\to\infty$ $\P$-a.s.}.
\]
The equivalence \eqref{equivsupremum} follows plainly by the time-changing relationship between $X$ and $\xi$, see \eqref{defn:X}.
\end{proof}
The statements in Theorem \ref{thm:speed} are thus equivalent if one replaces $X$ by the running supremum. We shall therefore work with $\bar{X}$. We treat the cases \ref{case1}, \ref{case3} and \ref{case2} separately in the subsections below.

\subsubsection{Case \ref{case1}: $\beta>\alpha>0$}
	
\begin{proof}[\bf Proof of Case \ref{case1} in Theorem \ref{thm:speed}]
\begin{lem}
\begin{equation}\label{eqninlem:spd:1-1}
\frac{T_{\infty}^{+}-T_{x}^{+}}{\varphi(x)}\bigg|_{\P_{1}(\cdot|T_{\infty}^{+}<\infty)}
\underset{x\rightarrow \infty}{\Longrightarrow}
\begin{cases} &\chi_{\alpha}^{\alpha-\beta}\times\varrho_{\alpha,\beta} \text{ if } \alpha<1\\
&1\text{ if } \alpha=1.
\end{cases}
\end{equation}
\end{lem}
\begin{proof}
We first carry out the computation  under $\P_{1}(\cdot|T_{a}^{-}=\infty)$, i.e. conditioning on level $a$ not being reached. The result under  $\P_{1}(\cdot|T_{\infty}^{+}<\infty)$ will be obtained by letting $a\to 0+$ in a second time.
Recalling the equivalences \eqref{eqn:explode1-3}. Under the assumptions of Theorem \ref{thm:speed},
for any $a\in(0,1)$ the following four events coincide a.s. under $\mathbb{P}_1$, 
\[
\{T_{a}^{-}=\infty,T_{\infty}^{+}<\infty\}
=\{\tau_{a}^{-}=\infty,\eta(\infty)<\infty\}
=\{\tau_{a}^{-}=\infty\}
=\{T_{a}^{-}=\infty\}.
\]
Set $g(z):=\mathbb{P}_z(\eta(\infty)/\varphi(z)>t|\tau_{a}^{-}=\infty)$. We know by Lemma \ref{lem+:wk} that $g(z)\rightarrow \mathbb{P}(\varrho_{\alpha,\beta}>t)$ as $z$ goes to $\infty$ for any $t>0$. Recall Proposition \ref{prop:overshoot} and the fact that it holds under $\mathbb{P}_{1}(\cdot|\tau_{a}^{-}=\infty)$ for every $x>0$. Applying the Markov property at time $\tau_{x}^{+}$ under $\mathbb{P}_{1}(\cdot|\tau_{a}^{-}=\infty)$ and recalling that $\varphi$ is regularly varying at $\infty$ with index $\alpha-\beta$, we have
\begin{align*}
\E_{1}\bigg[f\Big(\frac{\varphi(\xi(\tau_{x}^{+}))}{\varphi(x)}\Big); \frac{\eta(\infty)\circ \theta_{\tau_{x}^{+}}}{\varphi(\xi(\tau_{x}^{+}))}>t\bigg|\tau_{a}^{-}=\infty\bigg]
=&\ \E_{1}\bigg[f\big(\frac{\varphi(\xi(\tau_{x}^{+}))}{\varphi(x)}\big)\cdot g(\xi(\tau_{x}^{+}))\bigg|\tau_{a}^{-}=\infty\bigg]\\
&\underset{x\rightarrow \infty}{\longrightarrow} \mathbb{E}[f(\chi^{\alpha-\beta}_\alpha)]\cdot\P(\varrho_{\alpha,\beta}>t),
\end{align*}
for every continuous and bounded function $f$,
where $\theta_.$ denotes the shift operator.
We thus get
\[
\frac{T_{\infty}^{+}-T_{x}^{+}}{\varphi(x)}\bigg|_{\P_{1}(\cdot |T_{a}^{-}=\infty )}
=\frac{\varphi\big(\xi(\tau_{x}^{+})\big)}{\varphi(x)}\cdot
\frac{\eta(\infty)\circ\theta_{\tau_{x}^{+}}}{\varphi(\xi(\tau_{x}^{+}))}
\bigg|_{\P_{1}(\cdot |\tau_{a}^{-}=\infty )}
\underset{x\rightarrow \infty}{\Longrightarrow}
\chi_{\alpha}^{\alpha-\beta}\times\varrho_{\alpha,\beta}
\]
where $\chi_\alpha$ and $\varrho_{\alpha,\beta}$ are independent.

We now show that the convergence holds true for the process conditioned to explode. Recall that $\{T_\infty^{+}<\infty\}=\{\tau_0^{-}=\infty\}$, $\mathbb{P}_1$-a.s.. one has
\[\begin{aligned}
&\ \E_{1}\Big[f\Big(\frac{T_{\infty}^{+}-T_{x}^{+}}{\varphi(x)}\Big)\Big|T_{\infty}^{+}<\infty\Big]\\
=&\ \E_{1}\Big[f\Big(\frac{T_{\infty}^{+}-T_{x}^{+}}{\varphi(x)}\Big)\Big|T_{a}^{-}=\infty\Big]
\cdot\P_{1}\big(\tau_{a}^{-}=\infty\big|\tau_{0}^{-}=\infty\big)\
 + c_{1}\cdot\P_{1}\big(\tau_{a}^{-}<\infty\big|\tau_{0}^{-}=\infty\big)\\
=&\ \E_{1}\Big[f\Big(\frac{T_{\infty}^{+}-T_{x}^{+}}{\varphi(x)}\Big)\Big|T_{a}^{-}=\infty\Big]+ c_{2}\cdot \P_{1}\big(\tau_{a}^{-}<\infty\big|\tau_{0}^{-}=\infty\big)
\end{aligned}\]
where $|c_{1}|\le ||f||_{\infty}, |c_{2}|\le 2||f||_{\infty}$ and we make use of the fact that $\{\tau_{a}^{-}=\infty\}\subset\{\tau_{0}^{-}=\infty\}$.
Noticing that the limit of the first term above is independent of $a>0$ and the second term above is uniformly bounded by
\[
\P_{1}\big(\tau_{a}^{-}<\infty\big|\tau_{0}^{-}=\infty\big)
=\frac{e^{p(a-1)}(1-e^{-pa})}{1-e^{-p}}
=\frac{e^{p(a-1)}-e^{-p}}{1-e^{-p}}
\]
which tends to $0$ as $a$ goes to $0$, one thus have
\begin{equation}\label{eqn:spd:1-1}
\frac{T_{\infty}^{+}-T_{x}^{+}}{\varphi(x)}\bigg|_{\P_{1}(\cdot|T_{\infty}^{+}<\infty)}
\underset{x\rightarrow \infty}{\Longrightarrow}
\chi_{\alpha}^{\alpha-\beta}\times\varrho_{\alpha,\beta}.
\end{equation}
For the case $\alpha=1$, the limit in \eqref{eqn:spd:1-1} degenerates with $\chi_{\alpha}=\varrho_{\alpha,\beta}=1$.
\end{proof}
We now invert the residual time to explosion in order to study the supremum of the process.
\begin{lem}
One has as $t\to 0+$,
\begin{equation*}
\label{lemlinearspeed1}\frac{\bar{X}(T_{\infty}^{+}-t)}{\varphi^{-1}(t)}\bigg|_{\P_{1}(\cdot|T_{\infty}^{+}<\infty)}
\Rightarrow \begin{cases} &\frac{1}{\chi_{\alpha}}\times\varrho_{\alpha,\beta}^{\frac{1}{\beta-\alpha}}\text{ if } \alpha<1,\\
&1 \text{ if } \alpha=1.
\end{cases}
\end{equation*}
\end{lem}
\begin{proof}
For all small $t>0$ and large $x>0$,
\begin{equation}
\label{eqn:spd:1-3}
\{\bar{X}(T^{+}_{\infty}-t)>x\}
\subset\{T_{\infty}^{+}-T_{x}^{+}\ge t\}
\subset\{\bar{X}(T^{+}_{\infty}-t)\ge x\}.
\end{equation}
For $\lambda>0$, taking $x=\lambda\varphi^{-1}(t)$, and applying \eqref{eqn:spd:1-1} together with the fact that $\varphi\in\mathcal{R}_{\alpha-\beta}$,
we have
\[\begin{aligned}
\P_{1}\big(T_{\infty}^{+}-T_{x}^{+}\ge t\big|T_{\infty}^{+}<\infty\big)
=&\ \P_{1}\Big(\frac{T_{\infty}^{+}-T_{x}^{+}}{\varphi(x)}\ge \frac{t}{\varphi(x)}=\frac{\varphi(x/\lambda)}{\varphi(x)}\Big|T_{\infty}^{+}<\infty\Big)\\
\to&\ \P\big(\chi_{\alpha}^{\alpha-\beta}\times\varrho_{\alpha,\beta}\ge \lambda^{\beta-\alpha}\big)
\end{aligned}\]
as $t\to0+$ (consequently $y\to\infty$).
The result follows by substituting the limit above into \eqref{eqn:spd:1-3}, and the continuity of the limit distribution.

For the case $\alpha=1$,
\begin{equation}
\label{eqn:spd:1-2}
\frac{T_{\infty}^{+}-T_{x}^{+}}{\varphi(x)}\bigg|_{\P_{1}(\cdot |T_{\infty}^{+}<\infty)}\Rightarrow 1\quad\text{as $x\to\infty$}.
\end{equation}
For arbitrary $\varepsilon>0$, taking $x=(1+\varepsilon)\varphi^{-1}(t)$ in \eqref{eqn:spd:1-3}, we have
\[\begin{aligned}
&\ \limsup_{t\to0+}\P_{1}\big(\bar{X}\big(T_{\infty}^{+}-t\big)>(1+\varepsilon)\varphi^{-1}(t)\big|T_{\infty}^{+}<\infty\big)\\
\le &\ \limsup_{t\to0+}\P_{1}\bigg(\frac{T_{\infty}^{+}-T_{x}^{+}}{\varphi(x)}\ge \frac{t}{\varphi(x)}=\frac{\varphi(\varphi^{-1}(t))}{\varphi((1+\varepsilon)\varphi^{-1}(t))}\bigg|T_{\infty}^{+}<\infty\bigg)\\
\le &\ \limsup_{x\to\infty}\P_{1}\bigg(\frac{T_{\infty}^{+}-T_{x}^{+}}{\varphi(x)}\ge (1+\varepsilon)^{\frac{\beta-\alpha}{2}}\bigg|T_{\infty}^{+}<\infty\bigg)=0
\end{aligned}\]
by \eqref{eqn:spd:1-2} and the fact $\varphi\in\mathcal{R}_{\alpha-\beta}$. Similarly, taking $x=(1-\varepsilon)\varphi^{-1}(t)$ in \eqref{eqn:spd:1-3}, we have by \eqref{eqn:spd:1-2},
\[\begin{aligned}
&\ \liminf_{t\to0+}\P_{1}\big(\bar{X}\big(T_{\infty}^{+}-t\big)\ge(1-\varepsilon)\varphi^{-1}(t)\big|T_{\infty}^{+}<\infty\big)\\
\ge &\ \liminf_{t\to0+}\P_{1}\bigg(\frac{T_{\infty}^{+}-T_{x}^{+}}{\varphi(x)}\ge \frac{t}{\varphi(x)}=\frac{\varphi(\varphi^{-1}(t))}{\varphi((1-\varepsilon)\varphi^{-1}(t))}\bigg|T_{\infty}^{+}<\infty\bigg)\\
\ge &\ \liminf_{x\to\infty}\P_{1}\bigg(\frac{T_{\infty}^{+}-T_{x}^{+}}{\varphi(x)}\ge (1-\varepsilon)^{2(\beta-\alpha)}\bigg|T_{\infty}^{+}<\infty\bigg)=1,
\end{aligned}\]
 which proves the critical case $\alpha=1$. This finishes the proof.
\end{proof}
The proof of Theorem \ref{thm:speed}, Case \ref{case1}: $\beta>\alpha>0$ follows by combining the above Lemma with Lemma \ref{lem:equivsupremum}.
\end{proof}

\subsubsection{Case \ref{case2}: $1\ge\beta=\alpha>0$}\label{sec:b=a}
\begin{proof}[\bf Proof of Case \ref{case2} in Theorem \ref{thm:speed}]
Applying the Markov property at $\tau_{x}^{+}$, we have
\[
\frac{T_{\infty}^{+}-T_{x}^{+}}{\varphi(x)}\bigg|_{\P_{1}(\cdot |T_{a}^{-}=\infty )}
=\frac{\varphi\big(\xi(\tau_{x}^{+})\big)}{\varphi(x)}\cdot
\frac{\eta(\infty)\circ\theta_{\tau_{x}^{+}}}{\varphi(\xi(\tau_{x}^{+}))}
\bigg|_{\P_{1}(\cdot |\tau_{a}^{-}=\infty )}
\Rightarrow 1,
\]
where we use the fact that in this case $\varphi\in\mathcal{R}_{0}$ at $\infty$ and $y^{-1}\xi(\tau_{x}^{+})$ has a finite weak limit by Proposition \ref{prop:overshoot}.
The limit above then follows from Lemma \ref{lem+:wk2}.

The scaling limits under $\P_{1}(\cdot|T_{\infty}^{+}<\infty)$ and that for $X(T_{\infty}^{+}-t)$ are obtained similarly as in the proof of case $\alpha=1<\beta$ in \eqref{eqn:spd:1-2}, and details are therefore omitted.
\end{proof}

\subsubsection{Case \ref{case3}: $\beta>\alpha=0$ and proof of Proposition \ref{cor:a=0}}\label{sec=}
The arguments for $\beta>\alpha=0$ are different. In this case, $\chi_{\alpha}=\infty$ by Proposition \ref{prop:overshoot}
and $\varrho_{\alpha,\beta}$ has a standard exponential law; see Section \ref{sec:intro}. Therefore, the limiting  law in \eqref{eqn:spd:1-1} is degenerate at $0$ and the proof of the case $\alpha>0$ is not applicable.

We start by establishing the following lemma. Recall that $\Pi$ is the L\'evy measure of the ladder height process $H$ defined in \eqref{defn:ladder} and $\nu$ is that of $\xi$; see \eqref{defn:psi}.
\begin{lem}\label{lem:alpha=0} Assume that  $-\psi \in \mathcal{R}_0$ at $0$.
For every $x,s\to\infty$ such that either $\bar{\Pi}(s)/\bar{\Pi}(x)\to \lambda\in(0,1)$ for $0<p<\infty$ or  $\bar{\nu}(s)/\bar{\nu}(x)\to \lambda\in(0,1)$ for $p=\infty$, we have
\begin{equation}
\label{eqn:spd:2-2}
\P\big(\xi(\tau_{x}^{+})>s\big)\underset{x,s\to\infty}{\longrightarrow} \lambda.
\end{equation}
\end{lem}
\begin{proof}
If $\xi$ with $0<p<\infty$ is not a non-subordinator, by applying \cite[Theorem III.6]{Bertoin96book} to the ladder height process and \eqref{eqn:prop2:4}, when $-\psi\in\mathcal{R}_{0}$ at $0$, we have
\begin{gather}
\label{eqn:spd:2-1}
\P\big(\bar{\xi}(\tau_{x}^{+}-)>\delta y\big)\underset{y\rightarrow \infty}{\to} 0,
\quad\forall \delta\in(0,1)\\
\label{eqn:spd:2-4}
\text{ and } \P\big(\xi(\tau_{x}^{+})>s,\bar{\xi}\big(\tau_{x}^{+}-\big)\le\delta x\big)
= \int_{0}^{\delta}\mathcal{U}(x \ddr u)\bar{\Pi}\big(x(1-u)\big)\cdot \frac{\bar{\Pi}(s-xu)}{\bar{\Pi}(x(1-u))},
\quad\forall s>x.
\end{gather}
Applying \ktt\ to \eqref{eqn:prop2:2}, we get that
\begin{equation}\label{eqn:spd:2-3}
\hat{\bar{\Pi}}(\theta)\sim \frac{\psi(\theta)}{-p\theta}\in\mathcal{R}_{-1}
\Rightarrow
\int_{0}^{x}\bar{\Pi}(y)\ddr y\sim\frac{x\psi(1/x)}{-p}\in\mathcal{R}_{1}
\Rightarrow
\bar{\Pi}(x)\sim\frac{\psi(1/x)}{-p}\in\mathcal{R}_{0}
\end{equation}
as $\theta\to0+$ and $x\to\infty$, respectively, where the monotone density theorem is applied in the last identity (recall that $\bar{\Pi}$ defined in \eqref{eqn:prop2:2} is continuous and strictly decreasing on $(0,\infty)$).

For every $\lambda\in(0,1)$, let $s,x$ satisfy \eqref{eqn:spd:2-2}.
Since $\bar{\Pi}\in\mathcal{R}_{0}$, we have $s/x\to\infty$ and
\[
\frac{\bar{\Pi}(s-xu)}{\bar{\Pi}(x(1-u))}=\frac{\bar{\Pi}(s-xu)}{\bar{\Pi}(s)}\frac{\bar{\Pi}(x)}{\bar{\Pi}(x(1-u))}\frac{\bar{\Pi}(s)}{\bar{\Pi}(x)}\to\lambda
\quad\text{as $x\to\infty$},
\]
where the limit above holds uniformly for $u\in[0,\delta]$ by UCT.
Using \eqref{eqn:spd:2-4} and the above convergence gives
\[\begin{aligned}
& \left\lvert \P\big(\xi(\tau_{x}^{+})>s,\bar{\xi}(\tau_{x}^{+}-)\le \delta x\big)-\lambda\cdot\P\big(\bar{\xi}(\tau_{x}^{+}-)\le \delta x\big) \right\lvert\\
& \qquad \le \ \int_{0}^{\delta}\mathcal{U}(x \ddr u)\bar{\Pi}\big(x(1-u)\big)\cdot \Big|\frac{\bar{\Pi}(s-xu)}{\bar{\Pi}(x(1-u))}-\lambda\Big|\to0,
\end{aligned}\]
which further gives \eqref{eqn:spd:2-2} by applying \eqref{eqn:spd:2-1}.

The subordinator case ($p=\infty$) follows from the same arguments by working directly with $\xi$ and its L\'evy measure $\nu$ instead of $H$ and $\Pi$. We omit the details.
\end{proof}

We now proceed with the study of the speed of explosion.
\begin{proof}[\bf Proof of Case \ref{case3} in Theorem \ref{thm:speed}]
We are going to apply Lemma \ref{lem:alpha=0} with $x$ and $s$ chosen as certain functions of $t$, going to $\infty$, as $t\to0+$.
Applying the Markov property at $\tau_{x}^{+}$ to the process $\xi$ gives
\begin{equation}
\label{eqn:spd:2-5}
\begin{aligned}
&\ \P_{1}\big(T_{\infty}^{+}-T_{x}^{+}<t\big|T_{a}^{-}=\infty\big)
=\P_{1}\big(\eta(\infty)\circ\theta_{\tau_{x}^{+}}<t\big|\tau_{a}^{-}=\infty\big)\\
=&\ \int_{x}^{\infty}\P_{1}\big(\xi(\tau_{x}^{+})\in\,\ddr z\big|\tau_{a}^{-}=\infty\big)
\P_{z}\big(\eta(\infty)<t\big|\tau_{a}^{-}=\infty\big)\\
=&\ \int_{x}^{\infty}\P_{1}\big(\xi(\tau_{x}^{+})\in\,\ddr z\big|\tau_{a}^{-}=\infty\big)\P_{z}\Big(z>\varphi^{-1}\Big(t\Big/\frac{\eta(\infty)}{\varphi(z)}\Big)\Big|\tau_{a}^{-}=\infty\Big).
\end{aligned}
\end{equation}
Since $\varrho_{\alpha,\beta}$ for $\alpha=0$ is exponentially distributed, for arbitrary $\varepsilon>0$, there are $0<m<M<\infty$ so that for every $x$ large enough
\[
\P_{z}\Big(\frac{\eta(\infty)}{\varphi(z)}\in[m,M]\Big|\tau_{a}^{-}=\infty\Big)>1-\varepsilon,\quad\forall z>y.
\]
Plugging the above into \eqref{eqn:spd:2-5} and using the monotonicity of $\varphi^{-1}$, one can check that
\begin{equation}\label{eqn:spd:2-6}
\begin{aligned}
\P_{1}\big(T_{\infty}^{+}-T_{x}^{+}<t\big|T_{a}^{-}=\infty\big)
\le&\ \varepsilon+ \P_{1}\big(\xi(\tau_{x}^{+})>\varphi^{-1}(t/M)\big|\tau_{a}^{-}=\infty\big),\\
\P_{1}\big(T_{\infty}^{+}-T_{x}^{+}<t\big|T_{a}^{-}=\infty\big)
\ge&\ (1-\varepsilon)\cdot \P_{1}\big(\xi(\tau_{x}^{+})>\varphi^{-1}(t/m)\big|\tau_{a}^{-}=\infty\big).
\end{aligned}
\end{equation}
Let $x=\bar{\Pi}^{-1}(\bar{\Pi}(\varphi^{-1}(t))/\lambda)$
for fixed $\lambda\in(0,1)$ and $s_{1}=\varphi^{-1}(t/m), s_{2}=\varphi^{-1}(t/M)$ in \eqref{eqn:spd:2-2}.
Since $\bar{\Pi}\in\mathcal{R}_{0}$ and $\varphi\in\mathcal{R}_{-\beta}$, one can check that as $t\to0+$
\[
\varphi^{-1}\in\mathcal{R}_{-1/\beta},
\quad
\frac{\bar{\Pi}(s_{1})}{\bar{\Pi}(x)}
=\lambda\frac{\bar{\Pi}(\varphi^{-1}(t/m))}{\bar{\Pi}(\varphi^{-1}(t))}\to\lambda
\quad\text{and}\quad
\frac{\bar{\Pi}(s_{2})}{\bar{\Pi}(x)}
=\lambda\frac{\bar{\Pi}(\varphi^{-1}(t/M))}{\bar{\Pi}(\varphi^{-1}(t))}\to\lambda.
\]
Thus, applying Lemma \ref{lem:alpha=0}, \eqref{eqn:spd:2-2} holds true and one sees from  \eqref{eqn:spd:2-6} that
\[
\P_{1}\big(T_{\infty}^{+}-T_{x}^{+}<t\big|T_{a}^{-}=\infty\big)
\to\lambda
\quad\text{as $t\to0+$},
\]
which is continuous in $\lambda\in(0,1)$.
Similar to \eqref{eqn:spd:1-1}, one can also have
\[
\P_{1}\big(T_{\infty}^{+}-T_{x}^{+}<t\big|T_{\infty}^{+}<\infty\big)
\to\lambda
\quad\text{as $t\to0+$},
\]

Going back to the events in \eqref{eqn:spd:1-3}, we get by applying the latter result that for every $\lambda\in(0,1)$
\[\begin{aligned}
&\ \P_{1}\Big(\lambda\bar{\Pi}(\bar{X}(T^{+}_{\infty}-t))\ge \bar{\Pi}(\varphi^{-1}(t))\Big| T_\infty^{+}<\infty\Big)\\
=&\ \P_{1}\Big(\bar{X}(T_{\infty}^{+}-t)\le \bar{\Pi}^{-1}\big(\bar{\Pi}(\varphi^{-1}(t))/\lambda\big)\Big| T_\infty^{+}<\infty\Big)\underset{t\rightarrow 0+}{\to} \lambda,
\end{aligned}\]
which together with \eqref{eqn:spd:2-3} concludes the proof.
\end{proof}
Along establishing the nonlinear renormalisation for $\bar{X}(T^{+}_\infty-t)$ in the case $\alpha=0$ above, we have  implicitly addressed a nonlinear renormalisation of the overshoot of the parent L\'evy process $\xi$. We  further illustrate this by showing Proposition \ref{cor:a=0}.
\begin{proof}[\bf Proof of Proposition \ref{cor:a=0}] For every $\lambda\in(0,1)$, choose $s$ and $x$ such that $\bar{\Pi}(s)=\lambda \bar{\Pi}(x)$. Lemma \ref{lem:alpha=0} gives
\[\begin{aligned}
\P\big(\xi(\tau_{x}^{+})>s\big)
=&\ \P\big(\bar{\Pi}(\xi(\tau_{x}^{+}))\leq \bar{\Pi}(s)\big)\\
=&\ \P\big(\bar{\Pi}(\xi(\tau_{x}^{+}))\leq \lambda \bar{\Pi}(x)\big)
=\P\big(\bar{\Pi}(\xi(\tau_{x}^{+}))\big/\bar{\Pi}(x)\leq \lambda\big)\to\lambda
\end{aligned}\]
as $x\to\infty$. We apply further \eqref{eqn:spd:2-3} and the proof of the proposition is completed. The fact that the result also holds conditioned on $\{\tau_0^{-}=\infty\}$ follows from similar arguments as in the case $\alpha>0$ whose details are omitted.
\end{proof}

\noindent \textbf{Proof of Theorem \ref{thm:speed}:} To conclude, Theorem \ref{thm:speed} is obtained by combining the three cases treated above with Lemma \ref{lem:equivsupremum}. \qed
\\

In the last section, we present some observations on the classical CSBPs. The branching property allows one to study directly the explosion time and its renormalisation without going through a study of the perpetual integral.

\subsection{The case of CSBP}\label{sec:CSBP}
If the rate function $R$ is the identity function, the Markov process $X$ defined by \eqref{defn:X} is a classical CSBP with branching mechanism $\psi$. Namely, it fulfills the branching property: for any $x,y\in (0,\infty)$ and $t\geq 0$
$X_t(x+y)\overset{D}{=}X^1_t(x)+X^2_t(y)$,
where $X^1_t(x)$, $X^2_t(y)$ are independent copies of the process $X$ at time $t$ starting, respectively, from $x$ and $y$.  Notice that when $R$ is not linear, the branching property does not hold and the techniques below do not apply.

A fundamental consequence of the branching property lies on the fact that for any $x\in (0,\infty)$ and $t\geq 0$, the Laplace transform of the process at time $t$ satisfies  identity
\begin{equation}\label{cumulant}
\mathbb{E}_x(e^{-\lambda X_t})=e^{-xu_t(\lambda)}, \,\,\,\lambda>0
\end{equation}
with $t\mapsto u_t(\lambda)$ the solution to the following integral equation
\begin{equation}\label{cumulantode}
u_0(\lambda)=\lambda
\quad\text{ and }\quad
\int_{u_t(\lambda)}^{\lambda}\frac{\ddr u}{\psi(u)}=t.
\end{equation}
A lot of information is available from the identity \eqref{cumulant} and it is not in general necessary in this setting to go back to the representation of the process via a time-changed L\'evy process.
\subsubsection{Explosion criterion}
A striking feature in this setting is that the explosion of the process can be studied directly from \eqref{cumulant} as follows:
\begin{equation}\label{distributionexplosiontime}
\mathbb{P}_x(T_{\infty}^+>t)=\underset{\lambda \rightarrow 0}{\lim}\mathbb{E}_x[e^{-\lambda X_t}]=e^{-xu_t(0+)}
\end{equation}
with $u_t(0+):=\underset{\lambda \rightarrow 0}\lim u_t(\lambda)$. Hence, the process $X$ explodes with positive probability if and only if $u_t(0+)>0$. From \eqref{cumulantode}, one can verify that this is equivalent to the following integral condition
\begin{equation}\label{Dynkincriterion}
\int_0\frac{\ddr u}{|\psi(u)|}<\infty.
\end{equation}
More precisely, when \eqref{Dynkincriterion} holds, $t\mapsto u_t(0+)$ satisfies the equation
\begin{equation}\label{u(0)} \int^{u_t(0+)}_{0}\frac{\ddr u}{-\psi(u)}=t, \text{ for all } t\geq 0.
\end{equation}
The identity \eqref{distributionexplosiontime} provides in fact the cumulative distribution function of the explosion time. Recall $p$ the largest zero of $\psi$ and notice that $u_t(0+)\underset{t\rightarrow \infty}{\longrightarrow} p$.
Then
%
\begin{equation}\label{dfexplosiontime}
\mathbb{P}_x(T_\infty^{+}>t,T_\infty^{+}<\infty)=e^{-xu_t(0+)}-e^{-xp}.
\end{equation}

We now restudy briefly the speed of explosion and the renormalisation of the explosion time in the case $-\psi\in \mathcal{R}_\alpha$ at $0$ with $\alpha\in (0,1)$.
\subsubsection{The speed of explosion}
For simplicity, we assume from now on that $X$ has \textit{nondecreasing} sample paths, i.e. the parent L\'evy process $\xi$ is a subordinator. In this case $p=\infty$ and the explosion event is almost sure when \eqref{Dynkincriterion} holds. Using the identity $T^{+}_\infty=T^{+}_{a/u_t(0+)}+T^{+}_\infty \circ \theta_{T^{+}_{a/u_t(0+)}}$, the strong Markov property and \eqref{dfexplosiontime}, we get
\[\begin{aligned}
\mathbb{P}_1\big(u_t(0+)X_{T_\infty^{+}-t}>a\big)
&=\mathbb{P}_1\big(T_\infty^{+}-T^{+}_{a/u_t(0+)}\geq t\big)\\
&=\mathbb{E}_1\bigg[\mathbb{P}_{X_{T^{+}_{a/u_t(0+)}}}\big(T^{+}_\infty>t\big)\bigg]\\
&=\mathbb{E}_1\bigg[e^{-X_{T_{a/u_t(0+)}^{+}}u_t(0+)}\bigg].
\end{aligned}\]
Assume now that $-\psi\in \mathcal{R}_\alpha$ at  $0$ with $\alpha \in (0,1)$. By Proposition \ref{prop:overshoot},\[a\frac{u_t(0+)}{a}X_{T^+_{a/u_t(0+)}}=a \frac{u_t(0+)}{a}\xi_{\tau^+_{a/u_t(0+)}}\Longrightarrow a\chi_\alpha.\]
Therefore, for any $a>0$
$\mathbb{P}(u_t(0+)X_{T_\infty^{+}-t}>a)\underset{t\rightarrow 0+}{\longrightarrow} \mathbb{E}[e^{-a\chi_{\alpha}}]$.
We get finally the following limit in law
\begin{equation}\label{speedclassicalcsbp}
u_t(0+)X_{T_\infty^{+}-t} \Longrightarrow \mathbbm{e}_{\chi_{\alpha}}:=\frac{1}{\chi_{\alpha}}\mathbbm{e}_1 \text{ as } t \to 0+
\end{equation}
where $\mathbbm{e}_1$ is a standard exponential random variable independent of $\chi_{\alpha}$. In other words, $\mathbbm{e}_{\chi_{\alpha}}$ is a mixture of exponential random variables with random parameter given by $\chi_\alpha$.
\subsubsection{Renormalisation of the explosion time}
We now explain the link between \eqref{speedclassicalcsbp} and Case \ref{case1} of Theorem \ref{thm:speed}. First recall $\varphi$ in \eqref{defn:s}. By  change of variable $y=1/u$, we get $\varphi(x)=\int_0^{1/x}\frac{\ddr u}{-\psi(u)}$ for all $x$, and we easily verify with the help of \eqref{u(0)} that $\varphi^{-1}(t)=1/u_t(0+)$ for all $t\geq 0$. Since by assumption $-\psi\in \mathcal{R}_{\alpha}$ at $0$, $\varphi \in \mathcal{R}_{-(1-\alpha)}$ at $\infty$ and $\varphi^{-1} \in \mathcal{R}_{-1/(1-\alpha)}$ at $\infty$, then \[u_{t\varphi(x)}(0+)=\frac{\varphi^{-1}(\varphi(x))}{\varphi^{-1}(t\varphi(x))}\frac{1}{x} \underset {x\rightarrow \infty}{\sim} t^{\frac{1}{1-\alpha}}\frac{1}{x}.\]
%
Therefore, by applying \eqref{distributionexplosiontime}, we get
\begin{align}\label{renormalisationexplosiontime}
\mathbb{P}_x\left(\frac{T^+_\infty}{\varphi(x)}
\leq t\right)&=e^{-xu_{t\varphi(x)}(0+)}\underset{x\rightarrow \infty}{\longrightarrow} e^{-t^{\frac{1}{1-\alpha}}}.
\end{align}
Hence, we recover the result
$\frac{T^+_\infty}{\varphi(x)}\Longrightarrow  \rho_{\alpha,1} \text{ as } x \text{ goes to } \infty$
where $\rho_{\alpha,1}$ has a Weibull distribution with parameter $\frac{1}{1-\alpha}$. In particular, $\mathbbm{e}_1:=\rho_{\alpha,1}^{\frac{1}{1-\alpha}}$ is a standard exponential random variable and by \eqref{speedclassicalcsbp}
\[u_t(0+)X_{T_\infty^{+}-t}\Longrightarrow \mathbbm{e}_{\chi_\alpha}\overset{D}{=}\frac{1}{\chi_\alpha}\rho_{\alpha,1}^{\frac{1}{1-\alpha}}\]
which coincides with Case \ref{case1} of Theorem \ref{thm:speed} with $\beta=1$.

\begin{rmk}
\begin{enumerate}
\item
We mention that the convergence in \eqref{renormalisationexplosiontime} was noticed by Sagitov \cite{Sagitov}; see also Pakes \cite{Pakes} where some discrete nonlinear branching processes are studied.
\item The case $p<\infty$ for which the sample paths of the CSBP are not monotone could also be handled directly. But it requires to work with the process conditioned on explosion; see e.g. Fang and Li \cite{FL2019} for a study of the latter.
\end{enumerate}
\end{rmk}
\medskip
\noindent \textbf{Acknowledgements} Bo Li and Xiaowen Zhou are supported by NSERC (RGPIN-2016-06704).  Cl\'ement Foucart is supported  by the European Union (ERC, SINGER, 101054787). Views and opinions expressed are however those of the authors only and do not necessarily reflect those of the European Union or the European Research Council. Neither the European Union nor the granting authority can be held responsible for them.


\appendix
\section{Appendix}
This section is dedicated to the proofs of Lemmas \ref{lem+:uf} and \ref{lem=:uf}.
Recall that for the test function $f\in\mathcal{R}_{-\gamma}$ with $\gamma>0$ and $U_{a}f(x)<\infty$ for some $a>0$, we must have from Corollary \ref{cor:1},
\[
\int^{\infty}\frac{-f(y)}{y\psi(1/y)}\ddr y<\infty.
\]

In the non-subordinator case, i.e. $0<p<\infty$,
applying  \eqref{eqn:k:1}, \eqref{defn:Unonsub} and \eqref{eqn:f},
the finiteness of integral above implies
\[
\int^{\infty}\frac{f(y)\ddr y}{y}\int_{0}^{y}\kappa(z)\ddr z<\infty
\Rightarrow
\int^{\infty}\kappa(z)dz \int_{z}^{\infty}\frac{f(y)}{y}\ddr y<\infty
\Rightarrow
\int^{\infty}f(z)\kappa(z)dz<\infty
\]
where Fubini-Tonelli's theorem is applied. Thus, by \eqref{eqn:resl1}, for every  $x>a>0$
\begin{equation}\label{eqn:uf}
	\begin{aligned}
		U_{a}f(x)
		=&\ \int_{a}^{\infty}\E_{x}\big[f( \xi(t)); t<\tau_{a}^{-}\big] \ddr t
		=\int_{a}^{\infty}f(y)\big(\kappa(y-x)-e^{p(a-x)}\kappa(y-a)\big) \ddr y\\
		=&\ \int_{0}^{\infty}f(x+y)\kappa(y) \ddr y +\phi'(0)\int_{a}^{x}e^{p(y-x)}f(y) \ddr y-e^{p(a-x)}\int_{0}^{\infty}f(a+y)\kappa(y) \ddr y,
	\end{aligned}
\end{equation}
where $\kappa(y)=e^{py}\phi'(0)$ for $y<0$ in \eqref{eqn:resl1}.
\medskip

\noindent In the subordinator case, i.e. $p=\infty$, one has $\tau_{a}^{-}=\infty$ $\P_{x}$-a.s for $x>a>0$, and $U_{a}(x,\ddr y)=U(x,\ddr y)=\int_{0}^{\infty}\mathbb{P}_x(\xi(t)\in \ddr y)\ddr t$, see \eqref{eqn:potentialsub}.  One has $$U_{0}f(x)=\int_0^{\infty}f(y)U(x,\ddr y)=\int_0^{\infty}f(x+y)U(\ddr y).$$ Notice that the latter coincides with \eqref{eqn:uf} if we understand the  last two integrals to be zero (recall $p=\infty$) and replace $\kappa(y)\ddr y$ by $U(\ddr y)$.

\begin{proof}[\bf Proof of Lemma \ref{lem+:uf}]
We start with the non-subordinator case. It is straightforward to check from \eqref{eqn:uf} that $x\to f(\log x)\in\mathcal{R}_{0}$.
Therefore, applying \kara\ we have
\[\int_{a}^{x}e^{p(y-x)}f(y) \ddr y
=u^{-p}\int_{e^{a}}^{u}z^{p-1}f(\log z) \ddr z\Big|_{u=e^{x}}
\sim \frac{1}{p}f(x)\in\mathcal{R}_{-\gamma}\text{\ at\ }\infty,\]
which together with the fact $e^{-px}=o(f(x))$ gives
\begin{equation}\label{eqn+:uf1}
U_{a}f(x)-\int_{0}^{\infty}f(x+y)\kappa(y) \ddr y=O\big(f(x)\big).
\end{equation}


On the other hand, for $s>\alpha$, we have from Fubini-Tonelli's theorem that
\begin{equation}\label{eqn+:uf2}
\begin{aligned}
\ \int_{0}^{\infty}(x+y)^{-s}\kappa(y) \ddr y
=&\iint_{y,u>0}\frac{u^{s-1}}{\Gamma(s)}e^{-(x+y)u}\kappa(y) \ddr y\ddr u
= \int_{0}^{\infty}\frac{u^{s-1}}{\Gamma(s)}\hat{\kappa}(u)e^{-xu} \ddr u\\
\sim&\ \frac{\Gamma(s-\alpha)}{\Gamma(s)}x^{-s}\hat{\kappa}\big(1/x\big)
\sim \frac{\Gamma(s-\alpha)}{\Gamma(s)}\frac{-x^{-s}}{\psi(1/x)}\quad\text{as\ }x\to\infty,
\end{aligned}
\end{equation}
where \eqref{eqn:hk} and \ktt\ are used in the second line above. For any
$\varepsilon>0$, applying the UCT to $f(x)x^{s}\in\mathcal{R}_{s-\gamma}$ with $s\in(\alpha,\gamma)$, we have for $x$ large enough
\[
\Big|\frac{f(x+y)}{f(x)}-\big(\frac{x+y}{x}\big)^{-\gamma}\Big|\le \varepsilon\big(\frac{x+y}{x}\big)^{-s}
\quad\forall y>0.
\]
Thus, we have for such $x>0$
\[
\int_{0}^{\infty}\frac{f(x+y)}{f(x)}\kappa(y)\ddr y
=\int_{0}^{\infty}\frac{x^{\gamma}\kappa(y)\ddr y}{(x+y)^{\gamma}}
+\varepsilon_{1}(x)\cdot \int_{0}^{\infty}\frac{x^{s}\kappa(y)\ddr y}{(x+y)^{s}},
\]
for some $|\varepsilon_{1}(x)|<\varepsilon$.
Applying \eqref{eqn+:uf2} to the two integrals on the right hand side above,
one can find that they are of the same order, which implies
\[
\int_{0}^{\infty}f(x+y)\kappa(y) \ddr y
\sim \frac{-f(x)}{\psi(1/x)}\frac{\Gamma(\gamma-\alpha)}{\Gamma(\gamma)}
\sim \frac{\Gamma(\gamma-\alpha+1)}{\Gamma(\gamma)}\int_{x}^{\infty}\frac{-f(y)}{y\psi(1/y)}\ddr y
\in\mathcal{R}_{\alpha-\gamma}.
\]
Plugging into \eqref{eqn+:uf1} with the fact $\psi(0)=0$ completes the proof.

In the subordinator case, we omit the details since the proof follows from the same arguments. In particular \eqref{eqn+:uf2} holds true if one replaces $\kappa(y)\ddr y$ by $U(\ddr y)$ and $\hat{\kappa}(u)$ by $\hat{U}(u)$.
\end{proof}

\begin{proof}[\bf Proof of Lemma \ref{lem=:uf}]
In the non-subordinator case, by following the same argument as before, we see from \eqref{eqn:uf} that
 \eqref{eqn+:uf1} still holds. The lemma will be proved once we have shown that
\begin{equation}
\label{eqn=:uf1}
\int_{0}^{\infty}f(x+y)\kappa(y)\ddr y
\sim \int_{x}^{\infty}f(y)\kappa(y)\ddr y
\sim \frac{1}{\Gamma(\alpha)}\int_{x}^{\infty}\frac{-f(y)}{y\psi(1/y)}\ddr y
\in\mathcal{R}_{0}.
\end{equation}

Together with the finiteness of the integrals, we have from \eqref{eqn:k:1} and \eqref{eqn:f} that
\begin{gather}
\int_{x}^{\infty}\frac{-f(y)}{y\psi(1/y)}\ddr y
\sim
\Gamma(1+\alpha)\int_{x}^{\infty}\frac{f(y)}{y}\int_{0}^{y}\kappa(z)\ddr z\ddr y\in\mathcal{R}_{0},\label{eqn=:uf2}\\
\int_{x}^{\infty}\frac{f(z)}{z}dz\cdot\int_{0}^{x}\kappa(z)\ddr z
\sim \frac{f(x)}{\alpha}\int_{0}^{x}\kappa(z)\ddr z
=o\Big(\int_{x}^{\infty}\frac{f(y)}{y}\int_{0}^{y}\kappa(z)\ddr z\ddr y\Big),\label{eqn=:uf3}
\end{gather}
where \kara\ is applied in the last identity above, and entails further
\[\begin{aligned}
\int_{x}^{\infty}\frac{f(y)}{y}\int_{0}^{y}\kappa(z)\ddr z\ddr y
=&\ \int_{x}^{\infty}\frac{f(y)}{y}\ddr y\cdot\int_{0}^{x}\kappa(z)\ddr z
+\int_{x}^{\infty}\kappa(z)\int_{z}^{\infty}\frac{f(y)}{y}\ddr y\ddr z\\
\sim&\ \int_{x}^{\infty}\kappa(z)\int_{z}^{\infty}\frac{f(y)}{y}\ddr y\ddr z
\sim \frac{1}{\alpha}\int_{x}^{\infty}f(z)\kappa(z)\ddr z.
\end{aligned}\]
Plugging this into \eqref{eqn=:uf2} gives the second asymptotic equivalence in \eqref{eqn=:uf1}.

\medskip

For arbitrary $b>1$, we consider the difference
\[\begin{aligned}
&\ \int_{0}^{\infty}f(x+y)\kappa(y) \ddr y-\int_{bx}^{\infty}f(y)\kappa(y) \ddr y\\
=&\ \int_{bx}^{\infty}\Big(\frac{f(x+y)}{f(y)}-1\Big)f(y)\kappa(y)\ddr y+ \int_{0}^{bx}f(x+y)\kappa(y)\ddr y=I_{1}+I_{2}.
\end{aligned}\]
By the fact $\int_{x}^{\infty}f(y)\kappa(y)\ddr y\in\mathcal{R}_{0}$, we have
\[
I_{1}
\le \sup_{y>bx}\Big|\frac{f(x+y)}{f(y)}-1\Big| \int_{bx}^{\infty}f(y)\kappa(y)\ddr y
\sim\Big(1-\Big(\frac{b}{1+b}\Big)^{\alpha}\Big)\int_{x}^{\infty}f(y)\kappa(y)\ddr y.
\]
Moreover, for fixed $b>1$, $\frac{x+y}{x}\le b+1$ for $y<bx$ in $I_{2}$, applying UCT to $f\in\mathcal{R}_{-\alpha}$ one has
\[
I_{2}\sim f(x)\int_{0}^{bx}\Big(\frac{x}{x+y}\Big)^{\alpha}\kappa(y)\ddr y
\le f(x)\int_{0}^{bx}\kappa(y)\ddr y=o\Big(\int_{x}^{\infty}f(y)\kappa(y)\ddr y\Big),
\]
where \eqref{eqn=:uf3} and the fact $\int_{0}^{x}\kappa(y)\ddr y\in\mathcal{R}_{\alpha}$ are used. Therefore, first letting  $x\to\infty$ and then letting $b\to\infty$ gives \eqref{eqn=:uf1}, and this completes the proof.

Details are omitted in the subordinator case, this follows from the same arguments.
\end{proof}
\end{document}